\documentclass[11pt]{amsart}
\usepackage{color}

\usepackage[dvipdfmx]{graphicx}

\theoremstyle{plain}
\newtheorem{theorem}{Theorem}[section]
\newtheorem{proposition}[theorem]{Proposition}
\newtheorem{lemma}[theorem]{Lemma}
\newtheorem{corollary}[theorem]{Corollary}

\newtheorem{question}[theorem]{Question}

\theoremstyle{definition}

\newtheorem{remark}[theorem]{Remark}
\newtheorem{convention}[theorem]{Convention}

\newcommand{\ZZ}{\mathbb{Z}}

\newcommand{\RR}{\mathbb{R}}

\newcommand{\vv}{\vec{v}}
\newcommand{\ee}{\vec{\epsilon}}
\newcommand{\be}{\vec{e}}

\newcommand{\Fix}{\operatorname{Fix}}
\newcommand{\cl}{\operatorname{cl}}

\newcommand{\Isom}{\operatorname{Isom}}
\newcommand{\kernel}{\operatorname{Ker}}

\newcommand{\KK}{\mathcal{K}}
\newcommand{\hk}{\mathfrak{h}}

\numberwithin{figure}{section}

\newcommand{\figer}[4]{
\begin{figure}[h]
\begin{center}
\includegraphics[width=#2\textwidth]{#1}
\end{center}
\caption{#4 }
\label{#3}
\end{figure}
}

\begin{document}

\title[Equivariant Genera of strongly invertible knots]
{The equivariant genera of marked strongly invertible knots
associated with $2$-bridge knots}

\author{Mikami Hirasawa}
\address{Department of Mathematics\\
Nagoya Institute of Technology\\ 
Nagoya city 466-8555 Japan}
\email{hirasawa.mikami@nitech.ac.jp}

\author{Ryota Hiura}
\address{Inuyama-Minami High School\\
Hasuike 2-21, Inuyama City 484-0835 Japan}
\email{ryhiura@gmail.com}

\author{Makoto Sakuma}
\address{Osaka Central Advanced Mathematical Institute\\
Osaka Metropolitan University\\
3-3-138 Sugimoto, 
Osaka City 558-8585, Japan\newline
\indent Department of Mathematics\\
Hiroshima University\\
Higashi-Hiroshima City 739-8526, Japan}
\email{sakuma@hiroshima-u.ac.jp}

\makeatletter
\@namedef{subjclassname@2020}{%
  \textup{2020} Mathematics Subject Classification}
\makeatother

\subjclass[2020]{Primary 57K10, Secondary 57M60}

\thanks{
The first author is supported by JSPS KAKENHI Grant Number JP18K03296.
The third author is supported by JSPS KAKENHI Grant Number JP20K03614
and by Osaka Central Advanced Mathematical Institute 
(MEXT Promotion of Distinctive Joint Research Center Program JPMXP0723833165).
}

\begin{abstract}
A marked strongly invertible knot is a triple $(K,h,\delta)$
of a knot $K$ in $S^3$, a strong inversion $h$ of $K$,
and a subarc $\delta \subset \Fix(h)\cong S^1$ bounded by $\Fix(h)\cap K\cong S^0$.
An invariant Seifert surface for $(K,h,\delta)$ is an $h$-invariant Seifert surface 
for $K$ that intersects $\Fix(h)$ in the arc $\delta$.
In this paper, we completely determine the equivariant genus
(the minimum
of the genera of invariant Seifert surfaces for $(K,h,\delta)$) 
of every marked strongly invertible knot $(K,h,\delta)$
with $K$ a $2$-bridge knot.
\end{abstract}
\maketitle

\section{Introduction}
\label{sec:intro} 
A smooth knot $K$ in $S^3$ is said to be {\it strongly invertible}
if there is a smooth involution $h$ of $S^3$
which leaves $K$ invariant and fixes an unknotted loop intersecting $K$ in two points.
The involution $h$ is called a {\it strong inversion} of $K$.
As in \cite{HHS1, Sakuma1986},
we use the term {\it strongly invertible knot}
to mean a pair $(K,h)$ of a knot $K$ and a strong inversion $h$,
and regard two strongly invertible knots $(K,h)$ and $(K',h')$ {\it equivalent}
if there is an orientation-preserving diffeomorphism $\varphi$
of $S^3$ mapping $K$ to $K'$ such that $h'=\varphi h\varphi^{-1}$.

Recently, strongly invertible knots attract attention of various researchers
(see \cite[Sections 2 and 6]{HHS1} and references therein).
In particular, significant progresses 
on the equivariant $4$-genera of strongly invertible knots
have been made by
\cite{Alfieri-Boyle, Boyle-Issa2021a, Dai-Mallick-Stoffregen, Di-Prisa,
Di-Prisa-2, Di-Prisa-Framba, Miller-Powell}.
This paper is a sequel of \cite{HHS1}
which is devoted to the study of the ($3$-dimensional) 
equivariant genera 
of strongly invertible knots.
The ($3$-dimensional)
equivariant genera are actually defined for marked strongly invertible knots.

A {\it marked strongly invertible knot}
is a triple $(K,h,\delta)$,
where $(K,h)$ is a strongly invertible knot
and $\delta$ is a subarc of $\Fix(h)$ bounded by $\Fix(h)\cap K$.
Two marked strongly invertible knots $(K,h,\delta)$ and $(K',h',\delta')$ are regarded to be
{\it equivalent} if there is an orientation-preserving diffeomorphism $\varphi$
of $S^3$ mapping $K$ to $K'$ such that 
$h'=\varphi h\varphi^{-1}$ and $\delta'=\varphi(\delta)$.
By an  {\it invariant Seifert surface}
{\it for a marked strongly invertible knot} $(K,h,\delta)$,
we mean a Seifert surface $S$ for $K$
such that $h(S)=S$ and $\Fix(h)\cap S=\delta$.
Every marked strongly invertible knot $(K,h,\delta)$ admits an invariant Seifert surface
(see \cite{Boyle-Issa2021a, HHS1, Hiura}), and 
the {\it equivariant genus} 
$g(K,h,\delta)$ of $(K,h,\delta)$ is 
defined to be the minimum
of the genera of invariant Seifert surfaces for $(K,h,\delta)$.
An invariant Seifert surface for $(K,h,\delta)$ is said to be of 
{\it minimal equivariant genus}
if its genus is $g(K,h,\delta)$.

As for a periodic knot, namely, a knot $K$ which is 
preserved by a periodic rotation of $S^3$ whose axis is
disjoint from $K$, 
Edmonds \cite{Edmonds}  has shown that
every periodic knot admits an invariant minimal genus Seifert surface
(cf. Edmonds-Livingston \cite{Edmonds-Livingston}).
As observed in the previous paper \cite[Proposition 1.5]{HHS1}, 
it is also true for a marked strongly invertible knot $K$ 
if $K$ is a fibered knot.
However, in general $g(K,h,\delta)$ is not realized by 
a minimal genus Seifert surface for $K$.
Even so, the following two variants of Edmonds' theorem hold:

\begin{theorem}[{\cite[Theorem 1.3]{HHS1}}]  
\label{thm:disjoint-Seifert}
Let $(K,h)$ be a strongly invertible knot.
Then there is a minimal genus Seifert surface $F$ for $K$ 
such that $F$ and $h(F)$ have disjoint interiors.
\end{theorem}

\begin{theorem}[{\cite[Theorem 1.4]{HHS1}}]  
\label{thm:disjoint-Seifert2}
Let $(K,h,\delta)$ be a marked strongly invertible knot, and 
let $F$ be a minimal genus Seifert surface for $K$ 
such that $F$ and $h(F)$ have disjoint interiors.
Then there is an invariant Seifert surface $S$ for $(K, h, \delta)$
of minimal equivariant genus 
such that the interior of $S$ is disjoint from those of $F$ and $h(F)$.
\end{theorem}

In this paper, we use Theorem \ref{thm:disjoint-Seifert2} to
obtain our main result (Theorem \ref{thm:main}) which
determines the equivariant genera of 
all {\it marked strongly invertible $2$-bridge knots},
i.e., marked strongly invertible knots $(K,h,\delta)$
with $K$ a $2$-bridge knot.
We actually visualize invariant Seifert surfaces 
attaining equivariant genera.
Appendix presents a table of the
equivariant genera for $2$-bridge knots up to 10 crossings.

In the previous paper \cite[Theorem 1.2]{HHS1}, 
we proved that the gap between the genus 
$g(K)$ and the equivariant genus $g(K,h,\delta)$ can be arbitrarily large.
Our main Theorem \ref{thm:main} implies the following refinement of the result,
which says that
the genus $g(K)$ and the equivariant genus $g(K,h,\delta)$
are totally independent, except for the natural relations 
(i) $g(K)\le g(K,h,\delta)$ and 
(ii) $g(K,h,\delta)=0$ if $g(K)=0$.
(Here the second condition is due to Marumoto \cite[Proposition 2]{Marumoto}
(cf. \cite[Proposition 2.1(1)]{HHS1}).)

\begin{corollary}\label{cor:gap}
For any pair of positive integers $(g,\hat g)$ with $g\le \hat g$,
there is a marked strongly invertible knot $(K,h,\delta)$
such that $(g(K), g(K,h,\delta))=(g,\hat g)$. 
\end{corollary}

As for the behavior of the equivariant genera
of a strongly invertible knot $(K,h)$
for the two different choices of subarcs $\delta$ and
$\delta^c$ of $\Fix(h)$,
we have the following corollary
which says that the gap between the equivariant genera 
can be arbitrarily large.

\begin{corollary}\label{cor:gap2}
For any integer $d$,
there is a strongly invertible knot $(K,h)$
such that
$g(K, h, \delta)-g(K, h, \delta^c)=d$.
\end{corollary}

Both this paper and the preceding paper \cite{HHS1} 
are also motivated by our interest 
in the natural action of the symmetry group 
$\mathrm{Sym}(S^3,K)=\pi_0\mathrm{Diff}(S^3,K)$ on
the {\it Kakimizu complex} $MS(K)$ of a knot $K$.
The complex was introduced by Kakimizu \cite{Kakimizu1988}
as the {\it flag} simplicial complex 
whose vertices correspond to the (isotopy classes of) minimal genus Seifert surfaces for $K$
and edges to pairs of such surfaces with disjoint interiors.
In \cite{HHS1} we observed the following corollary of 
Theorem \ref{thm:disjoint-Seifert},
which may be regarded as a refinement of a special case of 
a theorem proved by Przytycki-Schultens \cite[Theorem 1.2]{Przytycki-Schultens}.

\begin{corollary}[{\cite[Corollary 1.6]{HHS1}}]
\label{HHS-Corollary 1.6}
Let $(K,h)$ be a strongly invertible knot,
and let $h_*$ be the automorphism of $MS(K)$
induced from of the strong inversion $h$.
Then one of the following holds.
\begin{enumerate}
\item
There exists a vertex of $MS(K)$
which is fixed by $h_*$.
\item
There exists an $h_*$-invariant edge of $MS(K)$
on which $h_*$ acts as the reflection in the 
central point of the edge.
\end{enumerate}
\end{corollary}

In the final section,
we describe the actions of the strong inversions 
on $MS(K)$ for every $2$-bridge knot
(Theorem \ref{thm:Action-Kakimizu}), which
illustrates Corollary \ref{HHS-Corollary 1.6}.

\medskip
Finally, we note that the recent work \cite{Di-Prisa-Framba} by Di Prisa and Framba
is also focused on the study of the strongly invertible $2$-bridge knots,
where it is proved that 
their $4$-dimensional equivariant genera are non-zero,
and moreover that 
they have infinite order in 
the equivariant cobordism group (with respect to any choice of 
\lq\lq direction'').
(To be precise, the paper is forcused on the study 
 of $(K(q/p), h_{q/p})$ (see Section \ref{sec:statement-result}).)

\medskip
This paper is organized as follows.
In Section \ref{sec:statement-result},
we recall the classification of the marked strongly invertible
$2$-bridge knots (Proposition \ref{prop:2-bridge-marked-strongly-invertible-knots}),
and state the main theorem (Theorem \ref{thm:main}) which determines their equivariant genera. 
In Section \ref{sec:construction},  
we give a systematic construction of invariant Seifert surfaces 
for marked strongly invertible $2$-bridge knots
which turn out to realize the equivariant genera.
In Section \ref{sec:proof},  
we use \lq\lq invariant sutured manifolds\rq\rq\ to give lower bounds for equivariant genera,
which turn out to be sharp.
In  Section \ref{sec:kakimizu}, 
we give a description of the actions of the strong inversions on the Kakimizu complex
of $2$-bridge knots.
Appendix presents a table of equivariant genera of 
marked strongly invertible $2$-bridge knots up to 10 crossings.

\section{Classification of marked strongly invertible $2$-bridge knots and the statement of the main result}
\label{sec:statement-result} 
Let $K(q/p)$ be the 2-bridge knot of slope $q/p$. Then naturally $p$ is odd.
Since $K((q \pm p)/p)$ is isotopic to $K(q/p)$, 
{\it we always assume $q$ is even and $1<|q|<p$.}
Then  $q/p$ admits a unique continued fraction expansion of the following format, 
with non-zero integer entries.
Following our previous paper \cite{HHS1},
we use this \lq\lq negative\rq\rq\ format 
which is different from 
the \lq\lq positive\rq\rq\ one used in other papers like \cite{Sakuma1986}  etc.

\begin{align*}\label{eq:cfe}
\frac{q}{p}
&=
           \cfrac{1}{2a_1
          -\cfrac{1}{2b_1
          -\cfrac{1}{2a_2
          -\cfrac{1}{\ddots
          -\cfrac{1}{2a_n
          -\cfrac{1}{2b_n
          }}}}}}
&=:
[2a_1,2b_1,\cdots, 2a_{n}, 2b_{n}]
\end{align*}

We use the following notation for describing $2$-bridge knots throughout this paper, except for Section \ref{sec:kakimizu}.
\[
K(q/p)=K[2a_1,2b_1, \cdots, 2a_n, 2b_n]
\]
For example, $K(2/3) =K[2,2]$ is the positive trefoil knot and 
$K(2/5) =K[2,-2]$ is the figure eight knot.
Note that it is well-known that
$K(q/p)$ has genus $n$
(see \cite{Hatcher-Thurston}).
A $2$-bridge knot is a torus knot only when
$\pm q=p-1$.
Note that 
$K((p-1)/p)=
K[2, 2, \dots, 2]$ is 
a torus knot of type $T(2, p)$ of genus $(p-1)/2$.
Other $2$-bridge knots are hyperbolic knots.

\medskip
By \cite[Corollary 2.3 and Remark 2.4]{HHS1}, the following holds.

\begin{proposition}
\label{prop:2-bridge-marked-strongly-invertible-knots_0}
Up to equivalence, the number of marked strongly invertible knots associated with $K(q/p)$
is equal to two or four according to whether
$\pm q=p-1$ or not.
\end{proposition}

Now we are going to see the four marked strongly invertible knots
associated to $2$-bridge knots $K(q/p)$.
(The case where $K(q/p)$ is a torus knot is included, but
the number 4 (of the marked strongly invertible knots)
is reduced to 2.)

For $K(q/p)$,
let $\beta$ be the $4$-braid  
$
\sigma_{2}^{2a_1} \sigma_{1}^{2b_1} 
\sigma_{2}^{2a_2}\sigma_{1}^{2b_2}
\cdots
\sigma_{2}^{2a_n} \sigma_{1}^{2b_n}$
described with Artin's standard generators $\sigma_1$ and  $\sigma_2$.
Then as in Figure \ref{fig:tauarcs1}(1), $K(q/p)$ is
depicted as the plat closure $\overline{\beta}$ of $\beta$.
As in Figure \ref{fig:tauarcs1}(2), we can also depict $K(q/p)$ 
as the plat closures of  another $4$-braid, $\beta'$, defined by
$$\beta'
=\sigma_{2}^{2a_1}\ (\sigma_{1}\sigma_{3})^{b_1}\ 
\sigma_{2}^{2a_2}\ (\sigma_{1}\sigma_{3})^{b_2}\  \cdots\ \ 
\sigma_{2}^{2a_n}\ (\sigma_{1}\sigma_{3})^{b_n},$$
together with the axis of $h_{q/p}$ and
the short and long arcs $\tau_{q/p}$ and $\tau_{q/p}^{c}$ of 
$\Fix(h_{q/p})$.
\figer{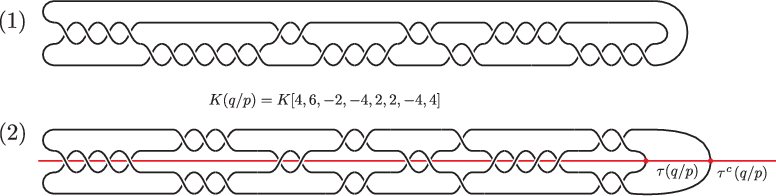}{1}{fig:tauarcs1}{$K(q/p)$ and its short and long involution arcs}
Since $K(q/p)$ has cyclic period $2$, 
it follows from \cite[Proposition 2.2(2)]{HHS1} that
$(K(q/p),h_{q/p},\delta)$ with $\delta\in\{\tau_{q/p},\tau^c_{q/p}\}$
are distinct marked strongly invertible knots.

\medskip
By \cite[Proposition 2.1(3)]{HHS1},
$K(q/p)$ admits one more strong inversion (see Figure \ref{fig:tauarcs2}).
To describe it,
note that  $K[2a_1,2b_1, \cdots, 2a_n, 2b_n]$ is isotopic to 
$K(q'/p)=K[2b_n, 2a_n, \cdots, 2b_1, 2a_1]$, where
$q'/p=[2b_n, 2a_n, \cdots, 2b_1, 2a_1]$.
(Here $q'$ is an even integer with $qq'\equiv 1$ $\pmod p$.)
By applying the preceding construction to this setting, 
we obtain the strong inversion $h_{q'/p}$ of $K(q'/p)$
and the short and long arcs $\tau_{q'/p}$ and $\tau_{q'/p}^{c}$ of 
$\Fix(h_{q'/p})$ as in Figure \ref{fig:tauarcs2}.

\figer{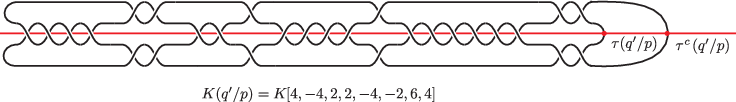}{1}{fig:tauarcs2}{$K(q'/p)$ and its short and long involution arcs}

By pulling them back through an isotopy between $K(q/p)$ and $K(q'/p)$,
we obtain a strong inversion of $K(q/p)$,
which we continue to denote by $h_{q'/p}$.
We also continue to denote the subarcs of
the fixed point set of the strong inversion $h_{q'/p}$ of $K(q/p)$,
obtained as the pullback of $\tau_{q'/p}$ and $\tau_{q'/p}^{c}$,
by the same symbols respectively. 

If $q\ne q'$, i.e., $q^2 \not\equiv 1\pmod{p}$,
then $K(q/p)$ is hyperbolic and
$\Isom^+(S^3\setminus K(q/p))$ (orientation-preserving isometry group)
is isomorphic to 
$\ZZ/2\ZZ\oplus \ZZ/2\ZZ$.
Thus, as described in \cite[Proposition 3.6 and its proof]{Sakuma1986},
$h_{q/p}$ and $h_{q'/p}$ are the two distinct strong inversions of $K(q/p)$.
See \cite[Fig. 3.2(1)]{Sakuma1986} or \cite[Fig. 2(1)]{ALSS},
where 
the two strong inversions  
are illustrated simultaneously in a single knot diagram.

\medskip
Finally, if $q= q'$, i.e., $q^2 \equiv 1\pmod{p}$,
then $\Isom^+(S^3\setminus K(q/p))$ is isomorphic to 
the dihedral group $D_8$ of order $8$.
The strong inversions $h_{q/p}$ and $h_{q'/p}$ are
equivalent, but 
another exceptional strong inversion $h'_{q/p}$ of $K(q/p)$ can be seen
as follows: 
When $q^2\equiv 1 \pmod p$, $q/p$ has a palindromic
continued fraction expansion, i.e., 
$q/p=
[2c_1,2c_2, \cdots, 2c_{n-1},2c_n,  2c_n,2c_{n-1}, \cdots, 2c_2,2c_1].$
In this case, $K(q/p)$ has a symmetric diagram
in which $\Fix(h'_{q/p})$ is
vertical to the projection plane as in Figure \ref{fig:palindromic}, where $K=K[2,-4,-4,2]$.
We can see that $K$ has such a diagram from the fact that
$K$ is the boundary of a minimal genus Seifert surface obtained by
a linear plumbing of unknotted annuli with $k$ half-twists where
$k=2, -4, -4, 2$ (see Figure \ref{fig:palindromicsurf}(1)).
The short vertical arc $\rho_{q/p}$ is in the central crossing ball, 
and the long complementary arc $\rho_{q/p}^c$ goes through infinity.
Then, 
by \cite[Proposition 2.2(2)]{HHS1}, 
$(K(q/p),h'(q/p),\delta)$ with $\delta\in\{\rho_{q/p},\rho_{q/p}^c\}$
are distinct marked strongly invertible knots.
See \cite[Fig. 3.2(2)]{Sakuma1986} or \cite[Fig. 3(1)]{ALSS},
where 
the two strong inversions are illustrated simultaneously in a single knot diagram.

\figer{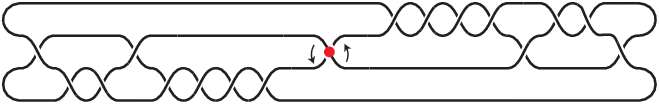}{0.55}{fig:palindromic}
{A palindromic diagram for $K[2,-4,-4,2]$, obtained as the boundary 
a linear plumbing on unknotted annuli}

In summary, we obtain the following proposition:

\begin{proposition}
\label{prop:2-bridge-marked-strongly-invertible-knots}
The marked strongly invertible knots associated with a nontrivial $2$-bridge knot  
$K=K(q/p)$ 
are classified as follows.

{\rm (1)} Suppose $\pm q=p-1$, i.e., $K$ is a torus knot.
Then $K$ has a unique strong inversion $h_{q/p}$ up to equivalence,
and 
there are precisely two marked strongly invertible knots
$(K,h_{q/p},\delta)$ with $\delta\in\{\tau_{q/p},\tau^c_{q/p}\}$
associated with $K$.

{\rm (2)} Suppose $q$ is an even integer such that $1<|q|<p$, i.e., 
$K$ is a hyperbolic~knot.
\begin{enumerate}
\item[{\rm (i)}]
Suppose $q^2\not\equiv 1 \pmod{p}$,
then $K$ has precisely two strong inversions $h_{q/p}$ and $h_{q'/p}$ 
up to equivalence,
where $q'$ is the unique even integer such that $qq'\equiv 1 \pmod{p}$ and $0<|q'|<p$.
Up to equivalence, there are precisely 
four marked strongly invertible knots associated with $K$,
as listed below.
\begin{align*}
(K,h_{q/p},\delta), & \quad \delta\in\{\tau_{q/p},\tau^c_{q/p}\},\\
(K,h_{q'/p},\delta), & \quad \delta\in\{\tau_{q'/p},\tau^c_{q'/p}\}
\end{align*}
\item[{\rm (ii)}]
Suppose $q^2\equiv 1 \pmod{p}$,
then $K$ has precisely two strong inversions 
$h_{q/p}$ and $h'_{q/p}$
up to equivalence.
Up to equivalence, there are precisely four 
marked strongly invertible knots associated with $K$,
as listed below.
\begin{align*}
(K,h_{q/p},\delta), & \quad \delta\in\{\tau_{q/p},\tau^c_{q/p}\},\\
(K,h'_{q/p},\delta), & \quad \delta\in\{\rho_{q/p},\rho^c_{q/p}\}
\end{align*}
\end{enumerate}
\end{proposition}

\begin{remark}
\label{rem:strong-equivalence}
{\rm In the case where  
$q^2\equiv 1 \pmod{p}$,
the number of the strong equivalence classes of strong inversions  
is bigger than the number $2$ of the equivalence classes of strong inversions.
Here two strong inversions are said to be {\it strongly equivalent}
if they are conjugate by a diffeomorphism $\varphi$ of $(S^3,K)$
which is pairwise isotopic to the identity.
In fact, in this case, there are $4$ strong inversions up to strong equivalence,
which are distinguished by their images in $Sym(S^3,K(q/p))\cong D_8$.
}
\end{remark}

Theorem \ref{thm:main} below is the main theorem of this paper which
completely determines the equivariant genera of 
marked strongly invertible knots for $2$-bridge knots.
Construction of invariant Seifert surfaces and 
lower bounds for equivariant genera are given in Sections \ref{sec:construction} and \ref{sec:proof}, respectively.
Corollaries \ref{cor:gap} and \ref{cor:gap2} obviously follow from Theorem \ref{thm:main}(2).

\begin{theorem}\label{thm:main}
Let $K=K(q/p)$ , $1\le |q|<p$
be a $2$-bridge knot with $g(K)=n$ denoted by
$K[2a_1,2b_1,\dots,2a_n,2b_n]$,
and let $q'/p=[2b_n,2a_n,\dots,2b_1,2a_1]$.

(1) Suppose $\pm q=p-1$, i.e., $K$ is a torus knot.
Then 
$K$ has two equivariant genera, which coincide with $g(K)$.
$$g(K, h,\delta)=g(K, h,\delta^c)=g(K)$$

(2) Suppose  $\pm q\neq p-1$, i.e., $K$ is a hyperbolic knot.

(2.1) If $q^2 \not\equiv 1 \pmod p$, 
$K$ has four equivariant genera as follows:
$$
\begin{array}{ccl}
g(K, h_{q/p},\tau_{q/p})&=&\sum_{i=1}^n |a_i| =g(K)+\sum_{i=1}^n (|a_i| -1)\\
g(K, h_{q/p},\tau_{q/p}^c)&=&g(K)+\#\{i |\  |b_i|>1\}=2g(K)-\#\{i |\  |b_i|=1\}\\
g(K, h_{q'/p},\tau_{q'/p})&=&\sum_{i=1}^n |b_i| =g(K)+\sum_{i=1}^n (|b_i| -1)\\
g(K, h_{q'/p},\tau_{q'/p}^c)&=&g(K)+\#\{i |\  |a_i|>1\}=2g(K)-\#\{i |\  |a_i|=1\}
\end{array}
$$

(2.2) If $q^2 \equiv 1 \pmod p$ (i.e., 
$[2a_1,2b_1,\dots,2a_n,2b_n]$ is palindromic),
then
$K$ has four equivariant genera as follows:
$$
\begin{array}{ccl}
g(K, h_{q/p},\tau_{q/p})&=&\sum_{i=1}^n |a_i| =g(K)+\sum_{i=1}^n (|a_i| -1)\\
g(K, h_{q/p},\tau_{q/p}^c)&=&g(K)+\#\{i |\  |b_i|>1\}=2g(K)-\#\{i |\  |b_i|=1\}\\
g(K, h'_{q/p},\delta)&=&g(K)\\
g(K, h'_{q/p},\delta^c)&=&g(K)
\end{array}
$$

\end{theorem}

\begin{remark}
\label{rem:main-torus}
The formulae in (2) hold even when $\pm q=p-1$.
In fact, if $\pm q=p-1$, then $|a_i|=|b_i|=1$ for all $i$ and
all quantities in the formulae in (2) become $g(K)$.
From (2.1), we see that there can be an arbitrarily large gap
between the genus and the equivariant genus if we use the
short invariant arc, but the gap can be at most $g(K)$ if we
use the long invariant arc.
\end{remark}

\begin{proof}[Proof of Theorem \ref{thm:main}]
(1) In this case, $K$ is a fibered knot, and the conclusion is given by
\cite[Proposition 1.5]{HHS1}.
Also, our construction of invariant Seifert surfaces 
in Section \ref{sec:construction}
yields the invariant fiber surfaces 
(whose genera are of course equal to $g(K)$),
for each choice of an invariant subarc of $\Fix(h)$.

(2)  
In Section \ref{sec:construction}, we construct an invariant Seifert surface
 for a given $(K(q/p), h, \tau)$. In Section  \ref{sec:proof},
 we give a lower bound for $g(K(q/p), h, \tau)$, which is
 equal to the genus of the Seifert surface constructed above.
 \end{proof}

\section{Construction of minimal genus invariant Seifert surfaces}
\label{sec:construction}

In this section, we construct invariant Seifert surfaces for
marked strongly invertible $2$-bridge knots $(K,h,\delta)$.
Some of them realizes $g(K)$, in which case we immediately see
$g(K,h,\delta)=g(K)$.
The rest of them having greater genera than $g(K)$ are proved
to realize $g(K,h,\delta)$ by the lower bounds given in Section \ref{sec:proof}.

\subsection{An invariant Seifert surface containing the short arc $\tau_{q/p}$}
Let $K=K(q/p)=K[2a_1,2b_1, \cdots, 2a_n, 2b_n]$,
$h=h_{q/p}$, $\tau=\tau_{q/p}$, and
let $k$ be one of the two subarcs of $K$ bounded by $K\cap \tau$.
Then $k\cup \tau$ spans an embedded disk 
$D$ (Figure \ref{fig:claspdisk}(1)),
and $D \cup h(D)$ is an immersed disk with clasp singularities  
spanned by $K$ (Figure \ref{fig:claspdisk}(2)).
Then by smoothing the singularities, 
we obtain an invariant Seifert surface $F$ for $(K, h, \tau)$
as in Figure \ref{fig:claspdisk}(3).
Since the number of the clasp singularities is $\sum_{i=1}^n |a_i|$, 
we have the following equalities:

$$g(F)=\sum_{i=1}^n |a_i| = g(K)+\sum_{i=1}^n (|a_i| -1)$$

\figer{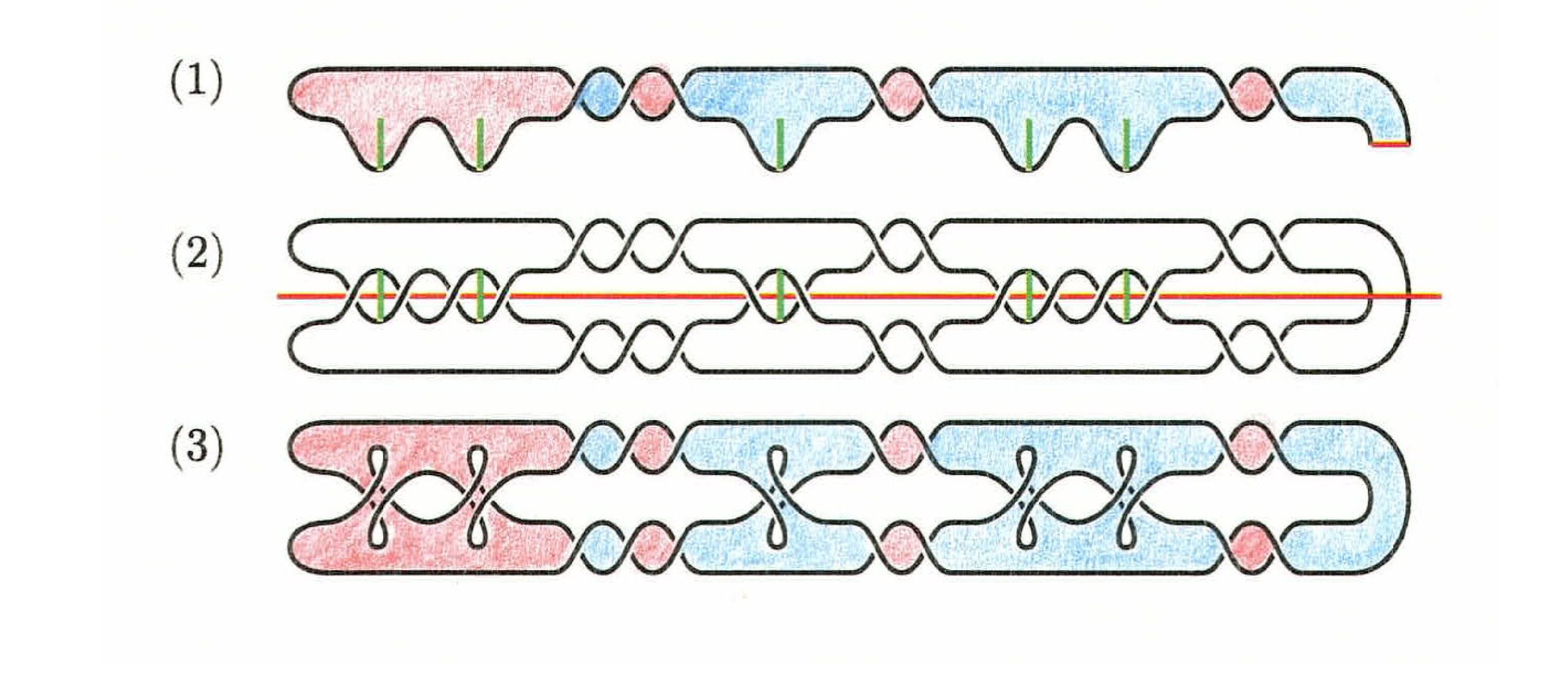}{1}{fig:claspdisk}{An invariant Seifert surface
for the short arc $\tau$}

\begin{remark}
In \cite{HHS1} cf. \cite{Hiura}, we gave 
an algorithm to construct an invariant Seifert surface
from a diagram where the involution axis lies in the 
projection plane. 
The surface in Figure \ref{fig:claspdisk} can be also obtained
by that method.
\end{remark}

\subsection{An invariant Seifert surface containing the long arc $\tau^c_{q/p}$}
We have two cases according to whether all $b_i$'s are even or not.

First, suppose that all $b_i$'s are even.
Then we span an invariant Seifert surface of genus 
$2n=2g(K)$ as in Figure \ref{fig:longarc}, which is constructed as follows:
First take a disk with $n$ holes.
Then for $i=1, 2, \dots, n$
add a band with $a_i$ twists to the $i^{\rm th}$ hole
and plumb a pair of
unknotted annuli with $b_{i}/2$ twists in a symmetric manner.
In \cite[Section 4]{HHS1} (cf.\,\cite{Hiura}) we 
proved that some of these surfaces realize the equivariant genera and 
showed that there can be
arbitrarily large gap between 
the genus and the equivariant genus.
The following is the completion of construction of
invariant Seifert surfaces for all marked $2$-bridge knots.
Next, if some $b_i$ is odd (Figure \ref{fig:b-odd}(1)), 
we twist the diagram (fixing the left end) at each site where
$b_i$ is odd as in Figure \ref{fig:b-odd}(2).
Then each band has an even number of twists.
Note that there are two ways in each twisting.
In particular, when some $b_i=\pm 1$, we twist the diagram 
at that site so that the resulting band has no twist at all.
As a conclusion, we can span an invariant Seifert surface $F$
as in Figure \ref{fig:b-odd}(3) 
according to if
$b_i$ is even, odd other than $\pm 1$, or  $\pm 1$.
Figure \ref{fig:b-odd} depicts the case where 
$(b_1,b_2,b_3,b_4,b_5)=(2,5,-2,-1,-2)$.
For the entries $b_1=2, b_3=-2$ and $b_5=-2$, 
we build bands
as in Figure \ref{fig:longarc}(2), and for the entry $b_2=5$ 
we build bands in a twisted way, and for the entry $b_4=-1$
we locally span the surface in a special way.
The genus of that surface $F$ is as follows:

$$g(F)=2g(K)-\#\{i \ \vert\  b_i=\pm 1\}=g(K)+\#\{i \ \vert\  |b_i|>1\}$$

\figer{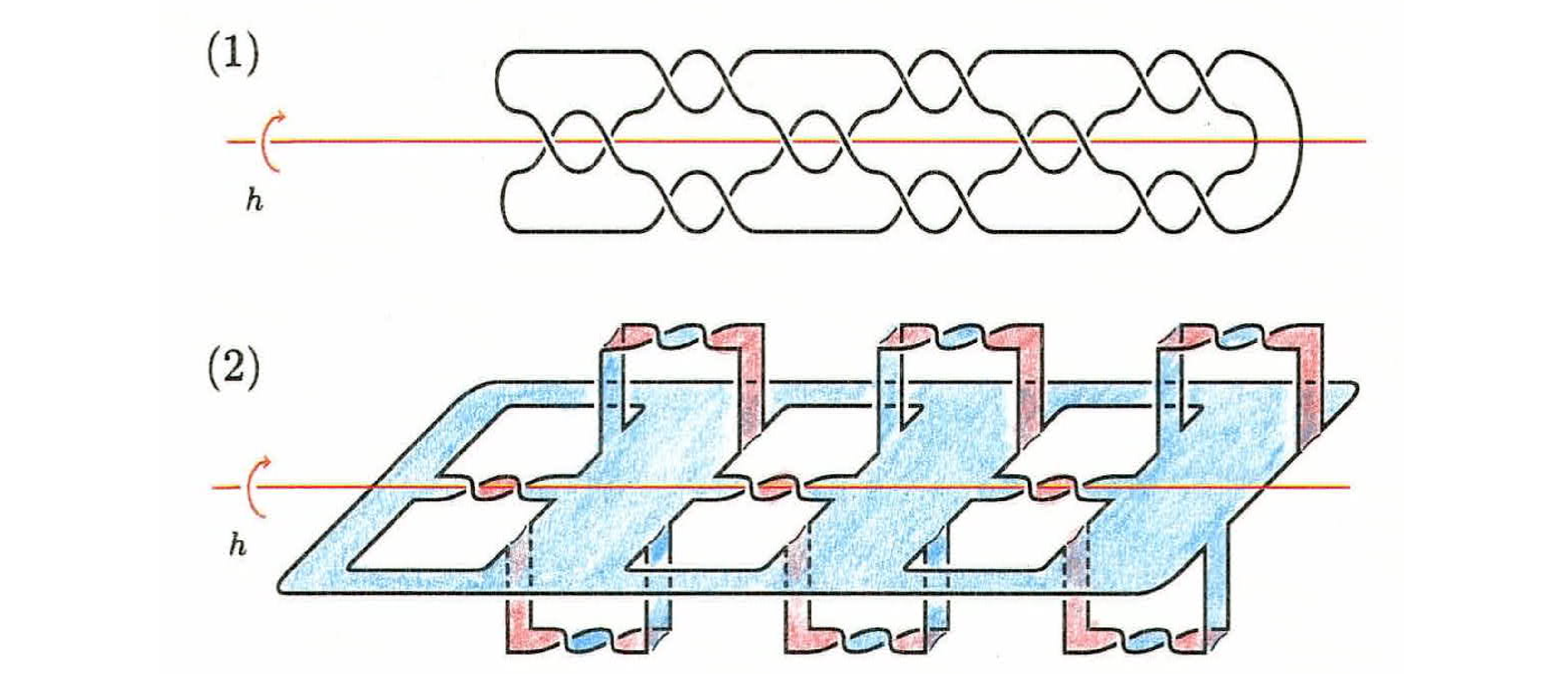}{1.1}{fig:longarc}{An invariant Seifert surface
for the long arc $\tau^c$}
\figer{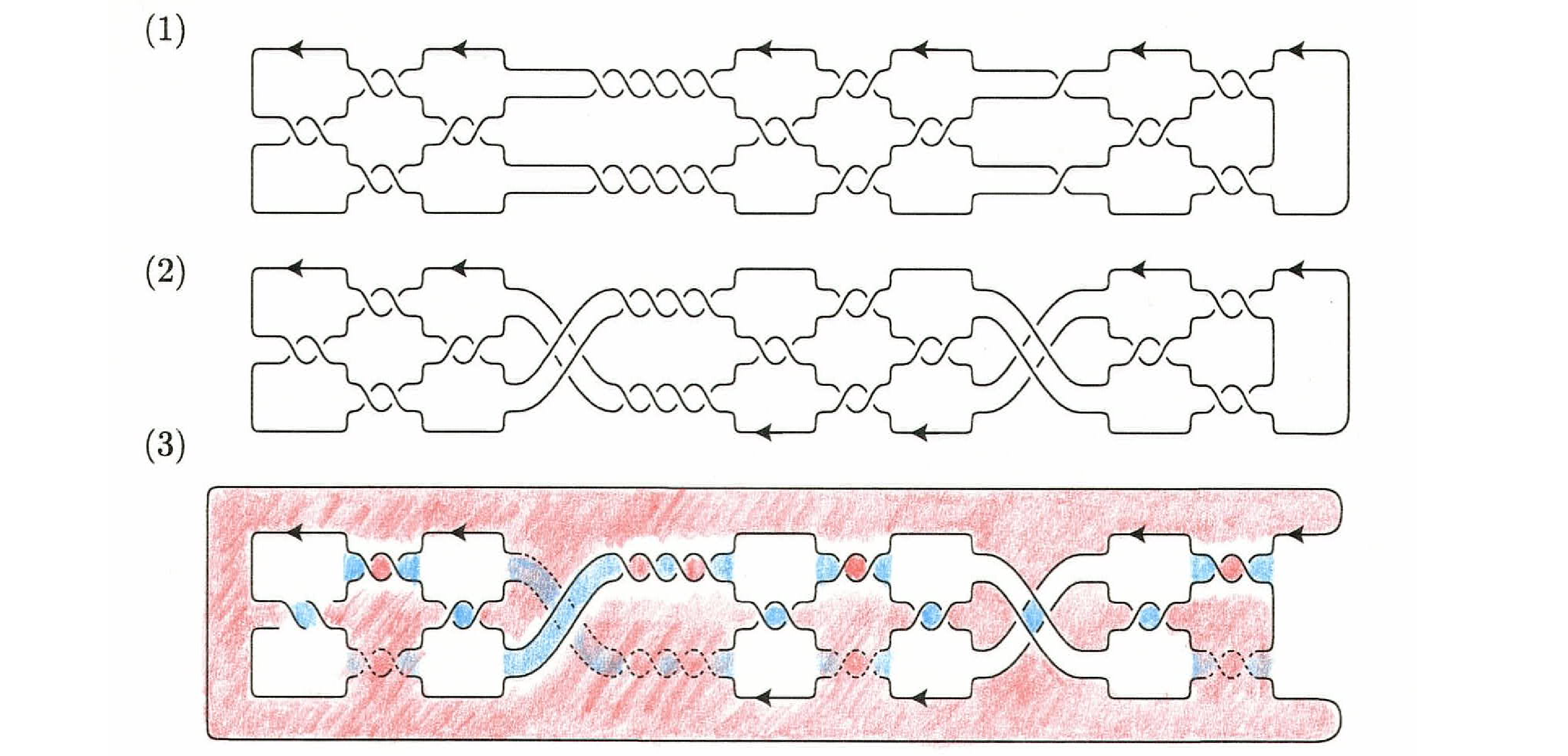}{1}{fig:b-odd}{Modification when some $b_i$ is odd}

\subsection{The exceptional case}
When $q^2\equiv 1 \pmod p$, $q/p$ has a palindromic
continued fraction expansion.
As described in Section \ref{sec:statement-result}, 
the exceptional strong inversion  $h'_{q/p}$ comes from the palindromicity, 
with the short and long vertical arcs 
$\rho_{q/p}$ and $\rho^c_{q/p}$.
Note that $K(q/p)$ has a minimal genus Seifert surface which is a linear
plumbing of unknotted twisted annuli.
From Figure \ref{fig:palindromicsurf}, 
we immediately see that $(K(q/p),h'_{q/p}, \rho_{q/p})$ and
$(K(q/p),h'_{q/p}, \rho^c_{q/p})$ have minimal genus
Seifert surfaces which are $h'_{q/p}$-invariant.

\figer{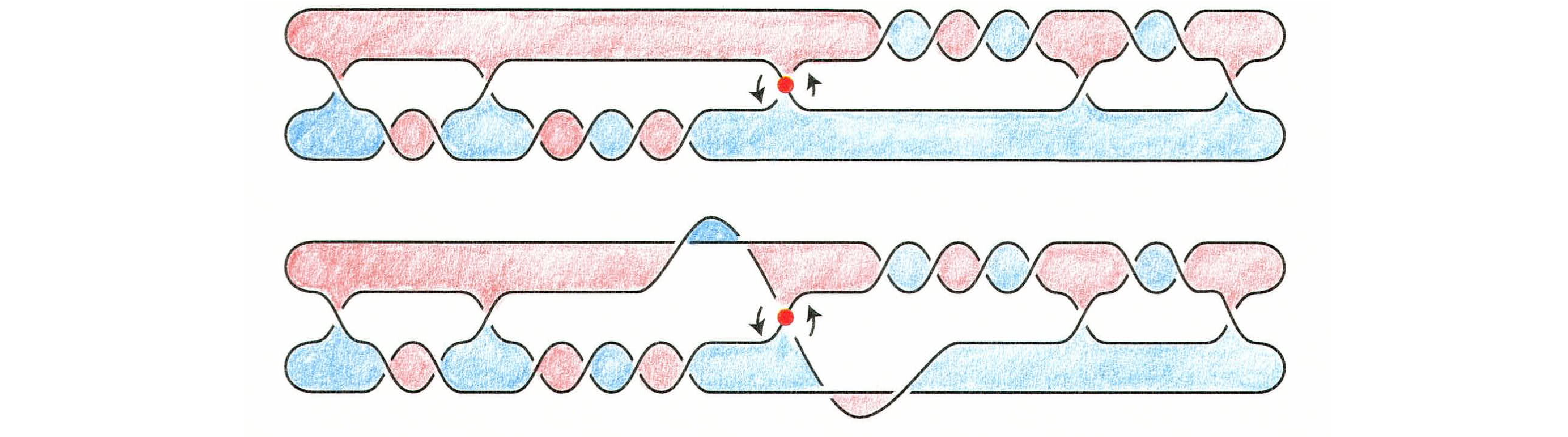}{1}{fig:palindromicsurf}{Invariant Seifert surfaces for the palindromic case}

\section{Sharp lower bounds for  equivariant genera}
\label{sec:proof} 
In this section, we give lower bounds for the equivariant genera of 
marked strongly invertible knots 
$(K(q/p), h_{q/p}, \tau_{q/p})$ and $(K(q/p), h_{q/p}, \tau_{q/p}^c)$. 
We build a sutured manifold between two minimal genus Seifert surfaces for
$K(q/p)$ which are exchanged by $h_{q/p}$, and then 
use Theorem \ref{thm:disjoint-Seifert2} to limit the existing range of 
invariant Seifert surfaces
of minimal equivariant genus for 
the marked strongly invertible knots.

\subsection{Sutured manifolds and product decompositions}
\label{subsec:sutured}
In this subsection, we first recall some of the terminology 
concerning sutured manifolds such as product decompositions.

A {\it sutured manifold} $(M,\gamma)$ 
defined by Gabai in \cite[Definition 2.6]{Gabai1983} is a
compact oriented $3$-manifold $M$ together with
a set $\gamma \subset \partial M$ of pairwise disjoint annuli $A(\gamma)$ and tori $T(\gamma)$.
In this paper, we only deal with sutured manifolds with
$\gamma$ consisting of only annuli.
Each component of $A(\gamma)$ contains a {\it suture}, i.e.,
a homologically nontrivial oriented simple closed curve.
The set of sutures is denoted by $s(\gamma)$.
Every component of $\cl(\partial M\setminus \gamma)$ is oriented,
so that the following condition is satisfied.
Let $R_+(\gamma)$ and $R_-(\gamma)$ 
be the subsurfaces of $\partial M$
whose normal vectors point out of or into $M$, respectively.
Then $\partial R_+(\gamma)$ and $\partial R_-(\gamma)$
are homologous to $s(\gamma)$ in $\gamma$.

\begin{convention}
A sutured manifold $(M,\gamma)$ is uniquely determined by 
the pair $(M,s(\gamma))$ and vice versa.
Therefore, throughout this paper,
we identity $\partial R_{\pm}(\gamma)$
with the closed up component of $\partial M\setminus s(\gamma)$,
and let $\gamma$ denote the suture $s(\gamma)$.
\end{convention}

A {\it product disk} in a sutured manifold $(M,\gamma)$ 
is a disk $\Delta$  properly embedded in $M$ 
such that $\partial \Delta \cap \gamma$ consists of 
two transversal intersection points,
i.e., $(\Delta,\Delta\cap \gamma)$ is a {\it bigon} in $(M,\gamma)$. 
Given a product disk 
$\Delta$ in $(M,\gamma)$,
we can produce a new sutured manifold $(M',\gamma')$
as illustrated in Figure \ref{fig:proddecomp}(1), 
by cutting $M$ along $\Delta$
and reconnecting the sutures naturally.
This operation is called the {\it product decomposition}
of $(M,\gamma)$ along the product disk $\Delta$.

In the rest of this section, we consider a surface $F$ embedded in
a sutured manifold $(M,\gamma)$ such that $\partial F = \gamma$
and $F\cap \Delta$ is an arc.
Then we  apply a product decomposition to  $(M, \gamma)$ cutting
$F$ along $F\cap \Delta$ (see Figure \ref{fig:proddecomp}(2)).

\figer{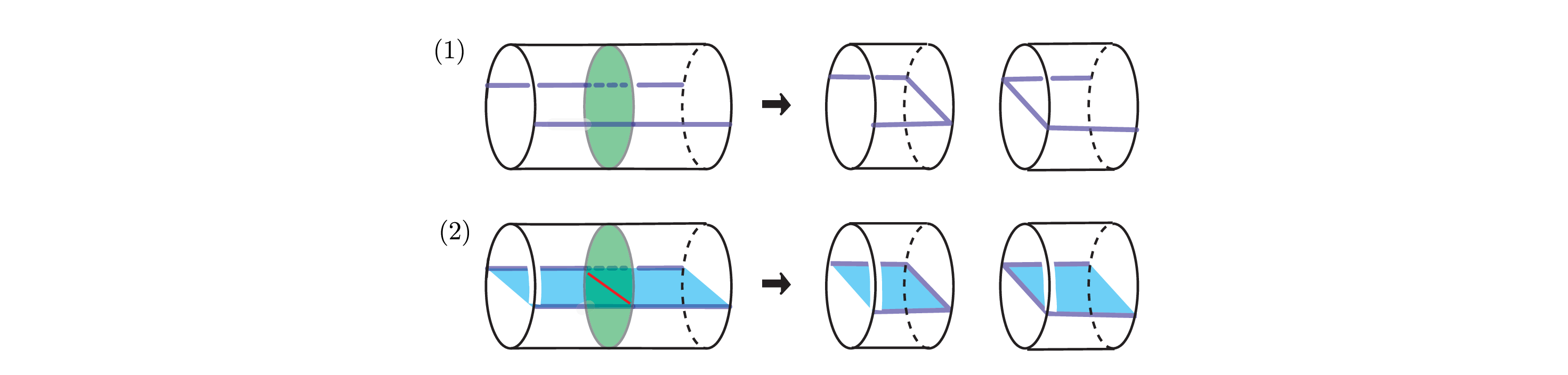}{1}{fig:proddecomp}{A product decomposition}

\subsection{Sutured manifolds between minimal genus Seifert surfaces}
\label{subsec:4.1}

As in Figure \ref{fig:sysSmfd}, we construct an 
$h$-invariant sutured manifold 
$(M, \gamma)$ embedded in $S^3$
such that $\gamma$ coincides with
$K(q/p)=$
$K[2a_1,2b_1,\cdots, 2a_n, 2b_n]$ in the form of
Figure \ref{fig:tauarcs1}(2)).
Actually, $(M, \gamma)$ is obtained as follows: 
(i) arrange $n$ $3$-balls penetrated by 
the involution axis, 
(ii)
connect each adjacent pair of balls
by a pair of $1$-handles,
(iii) attach a $1$-handle to the last ball,
(iv) drill out the balls along $\tau_{q/p}^c$, 
and finally
(v)
set up $\gamma$ which coincides with the knot $K(q/p)$.

\figer{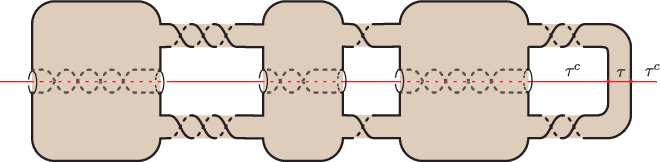}{0.7}{fig:sysSmfd}{The sutured manifold 
$(M, \gamma)$ between exchangeable minimal genus Seifert surfaces}

Note that
the subsurfaces $R_+(\gamma)$ and $R_-(\gamma)$ of $\partial M$
are minimal genus Seifert surfaces for $K(q/p)$ which are exchanged by the involution $h_{q/p}$.
(We will see in Theorem \ref{thm:Action-Kakimizu}(1-ii) that
this is 
the only mutually disjoint pair of minimal genus Seifert surfaces
that are exchanged by $h_{q/p}$.)

\subsection{Lowest genera of invariant Seifert surfaces for 
$(K(q/p), h_{q/p}, \tau_{q/p})$}
\label{subsec:estimate1}

As in Figure \ref{fig:cpd-in},  take $2n-1$ product disks for 
$(M, \gamma)$,
among which one is invariant under $h_{q/p}$ and the other 
$2n-2$ are
paired into those exchanged by $h_{q/p}$.
Apply $2n-1$ product decompositions to cut the $1$-handles so
that $(M, \gamma)$ is decomposed into $n$ sutured manifolds 
$(M_i,\gamma_i)$,
each of which is
the complementary sutured manifold of an unknotted annulus
with $a_i$ full-twists ($i=1,2,\dots, n$).
Each piece $(M_i, \gamma_i)$ appears as in 
Figure \ref{fig:surf-tau1}(1).

\figer{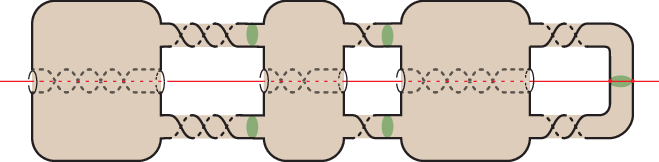}{0.7}{fig:cpd-in}{Product disks for 
$(M, \gamma)$ containing $\tau_{q/p}$}

\figer{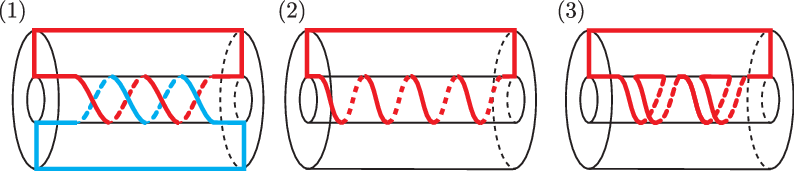}{0.7}{fig:surf-tau1}{
The sutured manifold $(M_i, \gamma_i)$ and its quotient
by $h_{q/p}$}

Since $\tau_{q/p} \subset M$, 
we see by Theorem \ref{thm:disjoint-Seifert2}  
that there is 
an invariant Seifert surface $F_\tau$ 
of minimal equivariant genus for
$(K(q/p), h_{q/p}, \tau_{q/p})$
contained in $(M,\gamma)$
so that  $\partial F_\tau=\gamma$ and 
$h_{q/p}(F_{\tau})=F_{\tau}$.
Let $\Delta$ be any of the  product disk taken above.
Then $F_{\tau} \cap \Delta$ contains an arc 
which is properly embedded both in $F_{\tau}$ and $\Delta$.
If $F_{\tau} \cap \Delta$ contains a loop, 
we can remove the innermost of such loops by
an $h_{q/p}$-equivariant surgery 
(in fact an $h_{q/p}$-equivariant isotopy), and therefore, 
we may assume that $F_{\tau}$ intersects each
product disk in an arc.
Hence, 
through the product decompositions above,
$F_{\tau}$ is cut into $F_i$'s,
where $F_i:=F_{\tau}\cap M_i$
is an $h_{q/p}$-invariant
orientable surface properly embedded in 
$M_i$ such that $\partial F_i=\gamma_i$
($1\le i\le n$).

\begin{lemma}
\label{lem:estimate-Betti}
We have $\beta_1(F_i)\ge 2|a_i|-1$.
\end{lemma}

\begin{proof}
Consider the manifold pair $(\check{M}_i,\check{\gamma}_i)$
that is obtained as the quotient of  $(M_i,\gamma_i)$ by $h_{q/p}$ 
(see Figure \ref{fig:surf-tau1}(2)). 
We identify  $(\check{M}_i,\check{\gamma}_i)$ with
$(S^1\times D^2, \kappa)$
where
$[\kappa]=2\alpha [S^1] + [\partial D^2]\in H_1(S^1\times \partial D^2)$.
Here $\alpha=a_i$, and we assume $\alpha>0$ without loss of generality.
Denote $F_i$ by $F$, and 
let $\check{F}=F/h_{q/p}$ be the image of $F$
in the quotient space $\check{M}_i=S^1\times D^2$.
Then $\check{F}$ is a compact surface properly embedded in 
$S^1\times D^2$ with $\partial \check{F} = \kappa$.
We show by induction on $\alpha$ that $\beta_1(\check{F})\ge \alpha$.
If $\alpha=1$, then $[\kappa]=2[S^1]\ne 0$ in $H_1(S^1\times D^2)$.
Hence $\check{F}$ is non-orientable, 
and so $\beta_1(\check{F})\ge 1=\alpha$.
Suppose $\alpha\ge 2$. We may assume $\check{F}$ intersects 
a meridian disk $D^2$ of $S^1\times D^2$ transversely.
By cut and paste operation, 
we may assume without increasing $\beta_1(\check{F})$
that $\check{F}\cap D^2$ is a non-empty union of disjoint arcs.
Pick an outermost component of $\check{F}\cap D^2$ in $D^2$,
and consider the surface $\check{F}'$ obtained from $\check{F}$ 
by cut and paste operation
along the outermost disk in $D^2$ bounded by the component.
Observe that $\partial \check{F}'$ is a simple loop
whose homology class (with a suitable orientation)
in $H_1(S^1\times \partial D^2)$ is  
$2(\alpha-1) [S^1] + [\partial D^2]$.
Hence, we have $\beta_1(\check{F}')\ge \alpha-1$ 
by the inductive hypothesis, and so
$\beta_1(\check{F})=\beta_1(\check{F}')+1\ge \alpha$.
On the other hand, 
since $F$ is a double covering of $\check{F}$, 
we have $1-\beta_1(F)=\chi(F) = 2\chi(\check{F})=2(1-\beta_1(\check{F}))$.
Hence we have 
$\beta_1(F)=2\beta_1(\check{F})-1\ge 2\alpha-1$.
\end{proof}

\begin{remark}
We can isotope Figure \ref{fig:surf-tau1}(2) so that it appears 
as in Figure \ref{fig:surf-tau1}(3).
Then we can see a (non-orientable) surface in 
$\check{M_i}$ bounded by $\check{\gamma}$
that is obtained from the vertical, square-shaped disk
by attaching $\alpha =|a_i|$ bands.
This surface has the first Betti number $\alpha$
and its inverse image in $M_i$ is an $h_{q/p}$-invariant
orientable surface bounded by $\gamma_i$
with the first Betti number $2 |a_i|-1$.
Thus  the inequality in Lemma \ref{lem:estimate-Betti}
is promoted to an equality. 
\end{remark}

Since $F_\tau$ is obtained from $\{F_i\}_{1\le i\le n}$ by
glueing along $2n-1$ arcs,
we have
$\beta_1 (F_{\tau})  =n+\sum_{i=1}^n \beta_1(F_i)$. 
Since  $\beta_1(F_i)\ge 2|a_i| -1$, we have the following conclusion:

$$\beta_1 (F_{\tau})  =n+\sum_{i=1}^n \beta_1(F_i)  \ge \sum_{i=1}^n 2|a_i|, 
{\rm \ and \  hence}$$
$$ g(F_{\tau})\ge \sum_{i=1}^n |a_i| =
g(K(q/p))+\sum_{i=1}^n( |a_i| -1).$$

See Figure \ref{fig:claspdisk}(3) for an invariant Seifert surface realizing the equality above.

\subsection{Lower bound for genera of invariant Seifert surfaces for 
$(K(q/p), h_{q/p}, \tau_{q/p}^c)$}
\label{subsec:estimate2}

Let $(N, \gamma)=(\cl(S^3-M),\gamma)$ be 
the complementary sutured manifold of $(M,\gamma)$.
Note that $(N, \gamma)$ contains $\tau_{q/p}^c$, and hence, 
by Theorem \ref{thm:disjoint-Seifert2},
there is 
an invariant Seifert surface 
$F_{\tau^c}$ of minimal equivariant genus
for $(K, h_{q/p},\tau_{q/p}^c)$
contained in $(N, \gamma)$.
As in Figure \ref{fig:cpd-out}, take $2n-1$ product disks in 
$(N, \gamma)$.

\figer{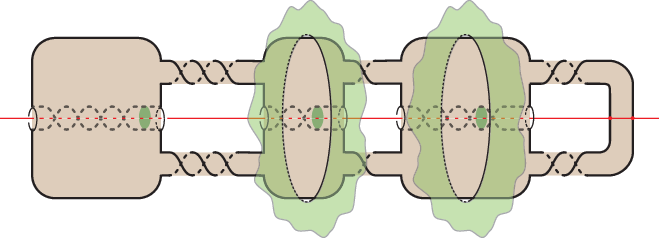}{0.7}{fig:cpd-out}{
$2n-1$ product disks for $(N,\gamma)$ containing $\tau_{q/p}^c$}

First, apply $n$ product decompositions (along the product disks in
the drilled holes) to 
fill in the drilled holes, and then $n-1$ product decompositions
(along the product disks whose central parts are omitted in 
Figure \ref{fig:cpd-out}) 
to decompose the result into $n$ sutured manifolds $(N_i, \gamma_i)$,
each of which is the 
complementary sutured manifold of the sutured manifold 
(that is shaded in Figure \ref{fig:out-pieces}(1))
obtained by thickening
an unknotted annulus with $|b_i|$ full-twists ($i=1,2,\dots, n$).
As in Subsection \ref{subsec:estimate1},
we may assume that $F_{\tau^c}$ intersects each of
the product disks in an arc.
Set $F_i:=F_{\tau^c}\cap N_i$.
Remark that $F_i$ is a surface properly embedded in $N_i$ with $\partial F_i=\gamma_i$,
where $N_i$ is the complement of the shaded handlebody in Figure \ref{fig:out-pieces}(1).

\figer{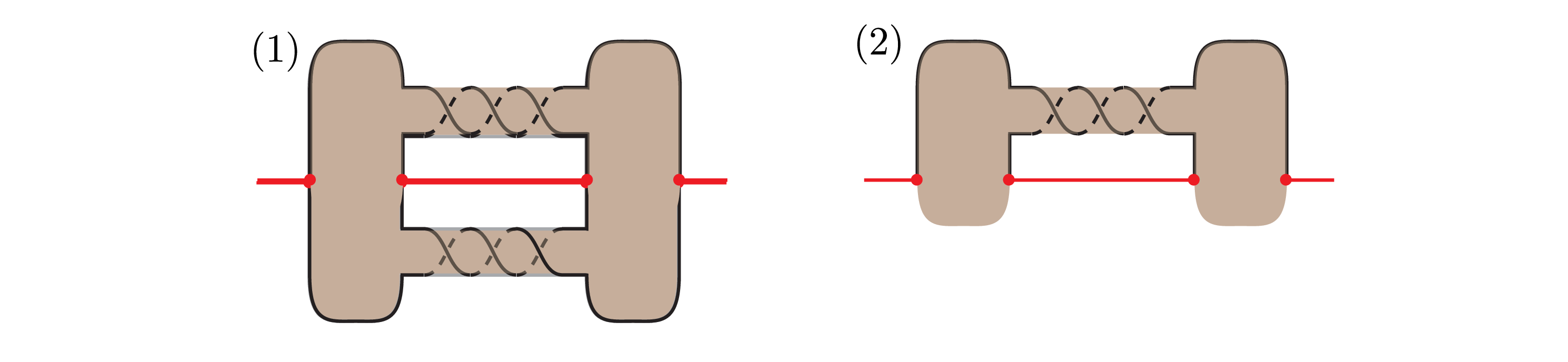}{0.9}{fig:out-pieces}{
Each of $N_i$ and $\check{N}_i:=N_i /h_{q/p}$ is 
the complement of the shaded handlebody.}

\begin{lemma}\label{lem:estimate-Betti2}
$\beta_1(F_i)\ge
\left\{ 
\begin{array}{ll}
3 & (|b_i|>1) \\
1 & (|b_i|=1)
\end{array} \right.
$
\end{lemma}

\begin{proof}
Consider the following quotient quadruple
$$(\check{N_i}, \check{\gamma_i}, \check{F_i},\check{\tau}^c_i):=
(N_i, \gamma_i, F_i,\tau^c_{q/p}\cap N_i)/h_{q/p}.$$
Note that $\check{N_i}$ is a ball,
and $\check{F_i}$ is a surface
embedded in the ball $\check{N_i}$ in such a way that
$\partial \check{F_i}$ consists of 
the two arcs
$\check{\gamma}_i$ in 
$\partial \check{N_i}$ and
 the two arcs $\check{\tau}^c_i$
 properly embedded in $\partial \check{N_i}$.
 Thus, 
 $\partial \check{F_i} =\check{\gamma_i} \cup \check{\tau}_i^c$
  (see Figure \ref{fig:out-pieces}(2)).
Now, $\partial \check{F_i}$ is a torus knot or link of type $T(2, b_i)$ and hence
$\beta_1(\check{F_i})\ge 1$ except for the trivial case of $|b_i|=1$, in which case,
$\partial \check{F_i}$ is the unknot.
By using the fact that $\Fix(h_{q/p}|_{F_i})$ consists of two arcs,
we see $1-\beta_1(F_i)=\chi(F_i)=2\chi(\check{F}_i)-2=-2\beta_1(\check{F}_i)$.
Therefore, $\beta_1(F_i)\ge2\beta_1(\check{F}_i)+1\ge3$
except for the primitive case $|b_i|=1$, in which case,
$\beta_1(\check{F}_i)\ge0$ and so $\beta_1(F_i)\ge~1$.
\end{proof}

\begin{remark}
 According to whether $|b_i|=1$ or $|b_i|>1$,
let $G_i$ be a disk or an unknotted annulus
with $|b_i|$ half-twists in $\check{M}_i$ bounded by
$\check{\mu}_i \cup\check{\tau}_i^c$.
Then the inverse image of $G_i$ is an 
$h_{q/p}$-invariant annulus or two-holed torus bounded
by $\mu_i$. 
Thus, the inequality in Lemma \ref{lem:estimate-Betti2}
is promoted to an equality.
\end{remark}

Since $F_{\tau^c}$ is obtained from $\{F_i\}_{1\le i\le n}$ by
glueing along $2n-1$ arcs,
we have
$\beta_1(F_{\tau^c})= n +\sum_{i=1}^n \beta_1(F_i)$,
and we have the following estimate from below
(see Figures \ref{fig:longarc} and \ref{fig:b-odd} for surfaces
 realizing the equality):

\begin{align*}
\beta_1(F_{\tau^c})&= n +\sum_{i=1}^n \beta_1(F_i)\\
& \ge n  + 3n - 2\#\{i\ |\ |b_i|=1\}, {\rm and\ hence,}\\
g(F_{\tau^c})&\ge2g(K(q/p))- \#\{i\ |\ |b_i|=1\}\\
&=g(K(q/p))+\#\{i\ |\ |b_i|>1\}
 \end{align*}
 
 Now we have obtained lower bounds for equivariant genera
 for all cases, which coincide with the upper bounds obtained
 in Section \ref{sec:construction} by explicit construction.
 This completes the proof of Theorem \ref{thm:main}.
 
\medskip
We end this section by posing the following variant of \cite[Question 6.1]{HHS1}.
\begin{question}
\label{question1}
{\rm
Is there an algebraic invariant of a marked strongly invertible knot 
that gives an effective estimate of the equivariant genus
and recovers the main result Theorem \ref{thm:main} of this paper?
}
\end{question}

\section{Actions of strong inversions on Kakimizu complexes}
\label{sec:kakimizu}

For an oriented link $L$ in $S^3$,
the {\it Kakimizu complex} $MS(L)$ of $L$ is the {\it flag} simplicial complex 
whose vertices correspond to the (isotopy classes of) 
minimal genus Seifert surfaces for $L$
and edges to pairs of such surfaces with disjoint interiors
(see \cite{Kakimizu1988}).
Kakimizu also introduced a similar complex $IS(L)$
whose vertex set is the set of incompressible Seifert surfaces,
and proved that both $MS(L)$ and $IS(L)$ 
are connected if $L$ is non-splittable,
refining the result of Scharlemann-Thompson \cite{Scharlemann-Thompson}
for minimal genus Seifert surfaces for knots.
Since then, various interesting results have been obtained
\cite{Agol-Zhang, Banks, Johnson-Pelayo-Wilson, Juhasz, Kakimizu2005,  Neel, Pelayo,
Przytycki-Schultens, Sakuma1994, Sakuma-Shackleton, Schultens, Valdez-Sanchez}.
In particular,
Przytycki-Schultens \cite{Przytycki-Schultens} proved 
the Kakimizu conjecture, which says that
$MS(L)$ is contractible for every non-splittable link $L$.
(To be precise, they prove the conjecture
in a more general setting with a \lq\lq rectified" definition of the complex,
where the two definitions coincide for oriented non-splittable links in $S^3$ 
all of whose minimal genus Seifert surfaces are connected.)
They also proved various interesting results concerning the fixed point sets
of the actions of subgroups of 
the symmetry groups
of non-splittable oriented links $L$ on (the rectified) $MS(L)$,
including a generalization of Edmonds' theorem.
Our Corollary \ref{HHS-Corollary 1.6} may be regarded as 
a slight refinement of a special case of (a variant of)
their theorem \cite[Theorems 1.2]{Przytycki-Schultens}.

In the remainder of this section, we study the case where $K$ is a $2$-bridge knot.
For a $2$-bridge knot $K=K(q/p)$,
the minimal genus Seifert surfaces for $K$ were classified by 
Hatcher-Thurston \cite[Theorem 1]{Hatcher-Thurston}
and the structure of $MS(K)$ was described 
by \cite[Theorem 3.3]{Sakuma1994}.
In this section, 
we describe the actions of the strong inversions on $MS(K)$. 
Throughout this section, we assume $q/p$ has the following continued fraction expansion (recall the convention introduced in 
Section \ref{sec:statement-result}).
\[
\frac{q}{p}=
[2a_1,2a_2,\cdots, 2a_{2n-1}, 2a_{2n}]
\]

For the above positive integer $n$,
let $\KK(n)$ be the simplicial complex characterized by the following properties
(see Figure \ref{fig:kcpx1}).
\begin{enumerate}
\item
The underlying space $|\KK(n)|$ is the $(2n-1)$-dimensional cube $[-1,1]^{2n-1}$ in the vector space $\RR^{2n-1}=\RR[\be_1,\be_2,\cdots,\be_{2n-1}]$.
\item
The vertex set of $|\KK(n)|$ is the set of corners 
$\{-1,1\}^{2n-1}$ of the cube.
Thus a vertex of $\KK(n)$ is identified with a vector $\ee=\sum_{i=1}^{2n-1}\epsilon_i \be_i$
with $\epsilon_i\in\{\pm 1\}$.
\item
The $(2n-1)$-simplices of $\KK(n)$ are described as follows.
Consider the $2n$ vectors $\vv_i:=(-1)^{i-1} \, 2(\be_{i-1}+\be_i)\in \RR^{2n-1}$ ($1\le i\le 2n$),
where $\be_{-1}:=\vec 0$ and $\be_{2n}:=\vec 0$.
Note that $\sum_{i=1}^{2n} \vv_i=\vec{0}$, which is the unique linear relation
among $\{\vv_i\}_{1\le i\le 2n}$ up to scalar multiplication.
Then $2n$ vertices of $\KK(n)$ span a $(2n-1)$-simplex if and only if
we can arrange them into a cyclically ordered set  $(\ee_0,\ee_1,\cdots,\ee_{2n-1})$
so that 
\[
\{\ee_{1}-\ee_{0},\ee_{2}-\ee_1,\cdots, \ee_{2n-1}-\ee_{2n-2}, \ee_0- \ee_{2n-1}\}
=
\{\vv_1,\vv_2,\cdots,\vv_{2n}\}.
\]
\end{enumerate}

We first recall the following special consequence of
\cite[Theorem 3.3]{Sakuma1994},
which also guarantees that $\KK(n)$ is a simplicial complex.

\begin{proposition}
\label{prop:structure-MS-special} 
If $|a_i|\ge 2$ for every $i$ ($1\le i\le 2n)$,
then $MS(K)$ is isomorphic to $\KK(n)$.
\end{proposition}

In order to explain Proposition \ref{prop:structure-MS-special},
let $T$ be a tree, with $2n$ vertices, 
whose underlying space is homeomorphic to a closed interval, 
and let $v_1, v_2, \ldots, v_{2n}$ be the vertices of $T$, 
lying on the interval in this order. 
For each vertex $v_i$ we associate an unknotted oriented annulus $A_i$ in $S^3$ with 
$a_i$ right-hand full twists. 
Then, $K(q/p)$ is equal to the boundary of a surface
obtained by successively plumbing the annuli 
$A_1, A_2, \ldots, A_{2n}$, 
and this surface is a minimal genus Seifert surface for $K$. 
Moreover, every minimal genus Seifert surface for $K$ is obtained in this way
by \cite[Theorem 1]{Hatcher-Thurston}.

There are $2^{2n-1}$ different ways of successive plumbing, 
according as $A_{i+1}$ is plumbed to $A_i$ from above or from below with respect to a normal vector field on $A_{j}$. 
Thus, successive plumbing can be represented by an {\it orientation of $T$}, directing each edge in one of two ways, by the following rule: 
If $\rho$ is an orientation of $T$, then we plumb $A_{i+1}$ to $A_i$ from above or below according as the edge joining $v_i$ and $v_{i+1}$ has initial point $v_i$ or $v_{i+1}$, respectively, with respect to $\rho$ 
(cf. \cite[Section 2]{Sakuma1994}).
We denote by $F(\rho)$ the Seifert surface for $K(q/p)$ determined by the orientation $\rho$.  
By \cite[Theorem 1]{Hatcher-Thurston} 
(cf. \cite[Theorem 2.3]{Sakuma1994}),
the condition that $|a_i|\ge 2$ for every $i$ implies that 
the correspondence $\rho \mapsto F(\rho)$ determines a bijection 
from the set $\mathcal{O}(T)$ of all orientations of $T$ 
to the vertex set of $MS(K)$.

To describe the structure of $MS(K)$, we introduce a few definitions. A vertex $v_{i}$ of $T$ is said to be a {\em sink} for the orientation $\rho$ of $T$ if every edge of $T$ incident on $v_{i}$ points towards $v_{i}$. If $v_{i}$ is a sink for $\rho$, then let $v_i(\rho)$ denote the orientation of $T$ obtained from $\rho$ by reversing the orientations of each edge incident on $v_i$. A {\em cycle} in $\mathcal{O}(T)$ is a sequence

\[
\rho_1 \xrightarrow{v_{i_1}} 
\rho_2 \xrightarrow{v_{i_2}} 
\ldots \xrightarrow{v_{i_{2n-1}}}
\rho_{n} \xrightarrow{v_{i_{2n}}}
\rho_1,
\] where $(i_1,i_2,\ldots, i_{2n})$ is a permutation of $\{1,2,\ldots,2n\}$ and $\rho_1,\rho_2,\ldots,\rho_{2n}$ are mutually distinct elements of $\mathcal{O}(T)$ such that $v_{i_{k}}(\rho_k)=\rho_{k+1}$ for every $k$, where the indices are considered modulo $n$. According to \cite[Theorem 3.3]{Sakuma1994},
$MS(K)$ can be described as follows:

\begin{itemize}
\item[$\circ$] 
The vertex set of $MS(K)$ is identified with $\mathcal{O}(T)$.

\item[$\circ$] 
A set of vertices $\{\rho_0,\rho_1,\ldots,\rho_k\}$ spans a $k$-simplex in $MS(K)$ if and only if it is contained in a cycle of $\mathcal{O}(T)$.
\end{itemize}
Moreover, $MS(K)$ gives a triangulation of the cube $[-1,1]^{2n-1}$
whose vertices are the corners of the cube 
(see \cite[Propositions 3.9 
and the paragraph after Remark 3.10]{Sakuma1994}).

Note that there are two special elements $\rho_+$ and $\rho_-$ of $\mathcal{O}(T)$
that are {\it alternating} in the sense that every vertex is either 
a sink or a source. 
We assume that $v_i$ is a sink or a source for $\rho_+$
according to whether $i$ is odd or even:
$\rho_-$ is obtained from $\rho_+$ by reversing the orientation
of every edge.
In \cite[Figure 2]{Hatcher-Thurston}
(under a suitable orientation convention),
$F(\rho_+)$ is constructed by plumbing 
of the bands $A_1, A_2, \cdots, A_{2n}$
where every plumbing disk $A_i\cap A_{i+1}$ 
is the horizontal plumbing square.
For $F(\rho_-)$,
every plumbing disk is 
the complement of the square in the 
horizontal plane containing it, compactified by a point
at $\infty$.
We can also observe that $F(\rho_{\pm})$ are the surfaces
$R_{\pm}(\gamma)$ in the boundary of the sutured manifold $(M,\gamma)$
with $\gamma=K$ introduced in Figure \ref{fig:sysSmfd}.

In order to show Proposition \ref{prop:structure-MS-special},
we introduce a notation which is a variant of that in 
\cite[the paragraph after Remark 3.10]{Sakuma1994}.
For $\rho\in\mathcal{O}(T)$, let 
$\ee=(\epsilon_1,\cdots, \epsilon_{2n-1})=\sum_{i=1}^{2n-1}\epsilon_i \be_i$ 
be the vertex of $\KK(n)$, determined by the rule that
$\epsilon_i$ is $+1$ or $-1$
according to whether 
$\rho_+$ and $\rho$ determine the same or distinct orientations
on the edge between $v_i$ and $v_{i+1}$.
Then the correspondence $\rho\mapsto \ee$ gives a bijection 
from $\mathcal{O}(T)$, the vertex set of $MS(K)$,
onto the vertex set of $\KK(n)$.
Moreover, we see that it induces an isomorphism between the flag simplicial complexes
$MS(K)$ and $\KK(n)$ by using the following fact.
Let  $\rho$ and $\rho'$ be distinct elements of $\mathcal{O}(T)$,
and let $\ee$ and ${\ee}\ '$ be the corresponding vertices of $\KK(n)$, respectively.
Then, there is a vertex $v_i$ of $T$
that is a sink for $\rho$ and $\rho'=v_i(\rho)$,
if and only if ${\ee}\ '=\ee + \vv_i$ for some $\vv_i$. 
Hence, we obtain Proposition \ref{prop:structure-MS-special}.

In the case where $|a_i|=1$ for some $i$,
$MS(K)$ is the quotient of $\KK(n)$ as described below
(see Figure \ref{fig:kcpx1}).

\figer{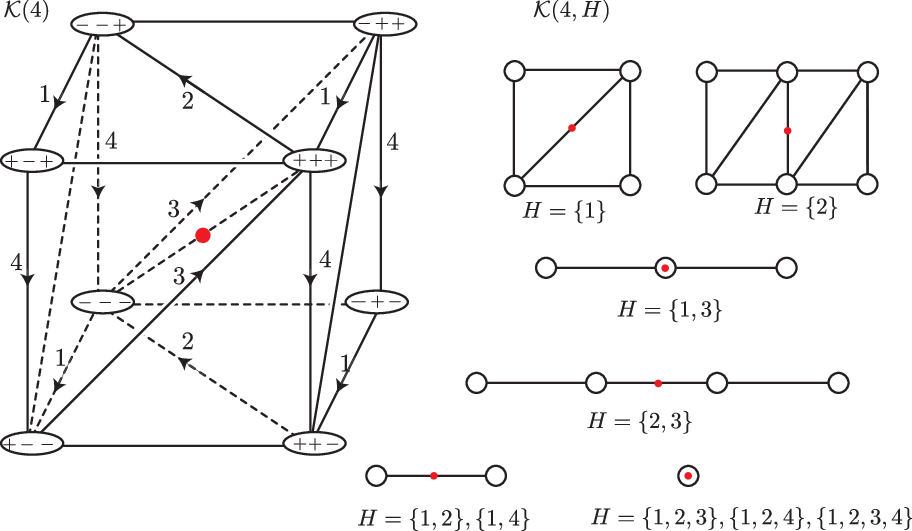}{1}{fig:kcpx1}{The complexes $\KK(n)$ (left) and $\KK(n,H)$ (right), representing the Kakimizu complex $MS(K(q/p))$.
In the left figure, 
the symbol $(+,+,-)$, for example, denotes the vertex of $\KK(4)$ 
corresponding to the vector $\ee=(+1,+1,-1)$.
An edge with an arrow and a number $i$ represents the vector $\vv_i$, where
$\vv_1=(2,0,0)$, $\vv_2=(-2,2,0)$, $\vv_3=(0,-2,2)$ and $\vv_4=(0,0,-2)$.
In all figures,
the red dot represents the unique fixed point of the involution $\hk$
induced by the strong inversion $h_{q/p}$.
If the fixed point is a vertex then it is represented by a circle with a red dot.}

Set 
\[
H=H(q/p)=\{i\in\{1,2,\cdots, 2n\} \ | \ |a_i|=1\},
\] 
where the symbol $H$ stands for the {\it Hopf band}.
Let $W=W(q/p)$ be the subspace of $\RR^{2n-1}$ spanned by the vectors
$\{\vv_i \ | \ i\in H\}$. 
Let $\pi=\pi_{q/p}$ be the projection from $\RR^{2n-1}$
to the quotient space $\RR^{2n-1}/W$.
Then there is a unique simplicial complex, $\KK(n;H)$,
such that 
the underlying space $|\KK(n;H)|$ is
the image $\pi(|\KK(n)|)$
and that $\pi$ induces a simplicial map from $\KK(n)$
onto $\KK(n;H)$.

By \cite[Theorem 3.3 and Proposition 3.11]{Sakuma1994},
we obtain the following.

\begin{proposition}
\label{prop:structure-MS}
For a $2$-bridge knot
$K=K(q/p)$ with
\[
\frac{q}{p}=
[2a_1,2a_2,\cdots, 2a_{2n-1}, 2a_{2n}],
\]
$MS(K)$ is isomorphic to $\KK(n;H)$ with $H=H(q/p)$.
\end{proposition}

\begin{remark}
(1) In the above proposition,
the vertex of $MS(K)\cong \KK(n;H)$ 
determined by $\ee\in \{-1,1\}^{2n-1}$
is represented by the Seifert surface $F(\rho)$,
where $\rho$ is the element of $\mathcal{O}(T)$
corresponding to $\ee$ through the bijection
between $\mathcal{O}(T)$ and $\{-1,1\}^{2n-1}$ described
at the end of the explanation of 
Proposition \ref{prop:structure-MS-special}.

(2) Except when $H$ is equal to the whole set $\{1,2,\cdots, 2n\}$, 
we have $\dim |MS(K)|=2n-1-(\# H)$.
The exceptional case occurs 
if and only if $K$ is a fibered knot.
\end{remark}

We now describe the action of the strong inversions on $MS(K)$ with $K=K(q/p)$.
To this end, let $\tilde \hk$ and $\tilde \hk'$, respectively, be the linear automorphisms 
of $\RR^{2n-1}$
represented by the diagonal matrix 
$(-\delta_{i,j})_{1\le i,j\le 2n-1}$ and the anti-diagonal matrix
$(\delta'_{i,j})_{1\le i,j\le 2n-1}$,
where $\delta_{i,j}$ is Kronecker's delta,
and $\delta'_{i,j}=1$ or $0$ 
according to whether $i+j=2n-1$ or not.
Note that 
(1)
the subspace $W=W(q/p)$ is preserved by $\tilde \hk$ 
and that 
(2)
$W$ is preserved by $\tilde \hk'$ if and only it it is
{\it symmetric}
in the following sense: for any pair of positive integers $i$ and $j$ with $i+j=2n$,
we have $i\in H$ if and only if $j\in H$.
Now we define automorphisms $\hk$ and $\hk'$ 
of the simplicial complex $\KK(n;H)$
as follows (see Figures \ref{fig:kcpx1} and \ref{fig:kcpx2}).

\begin{enumerate}
\item
The linear automorphism $\tilde \hk$ descends to a linear automorphism of $\RR^{2n-1}/W$,
and its restriction to $|\KK(n;H)|$ determines the automorphism
$\hk$ of $\KK(n;H)$.
\item
Suppose $H$ is symmetric.
Then the linear automorphism $\tilde \hk'$ descends to a linear automorphism of $\RR^{2n-1}/W$,
and its restriction to $|\KK(n;H)|$ determines the automorphism
$\hk'$ of $\KK(n;H)$.
\end{enumerate}
Then we have the following theorem.

\begin{theorem}
\label{thm:Action-Kakimizu}
For a $2$-bridge knot
$K=K(q/p)$ with
\[
\frac{q}{p}=
[2a_1,2a_2,\cdots, 2a_{2n-1}, 2a_{2n}],
\]
the actions of strong inversions of  $K$ on $MS(K)$
are described as follows.

{\rm (1)}
The automorphisms of $MS(K)\cong \KK(n;H)$ induced by
the strong inversions $h_{q/p}$ and $h_{q'/p}$ of $K$
are both equal to the automorphism $\hk$.
Moreover, the following hold,
where $F_+:=F(\rho_+)=F(+1,\cdots, +1)$
and $F_-=F(\rho_-)=F(-1,\cdots,-1)$.
\begin{enumerate}
\item[{\rm(i)}]
Suppose $|a_i|=1$ either for all odd $i$ or for all even $i$,
namely, $H=H(q/p)$ contains either all odd $i$ or all even $i$.
Then $F_+$ and $F_-$ represent the same vertex of $MS(K)$,
and it is the unique fixed point of $\hk$.
\item[{\rm(ii)}]
Suppose the above condition does not hold,
i.e., $|a_i|\ge 2$ for some odd $i$ and also for some even $i$.
Then $F_+$ and $F_-$
represent distinct vertices of $MS(K)$,
and they span an edge $e$ in $MS(K)$.
The center of $e$ is the unique fixed point of $\hk$,
and $e$ is the unique edge preserved by $\hk$.
\end{enumerate}

{\rm (2)}
Suppose $q^2\equiv 1 \pmod{p}$.
Then $H=H(q/p)$ is symmetric, and 
the automorphism of $MS(K)\cong \KK(n;H)$ induced by
the exceptional strong inversion $h'_{q/p}$
is equal to $\hk'$.
In particular, $\Fix(\hk')$ is a ball properly embedded in 
$|\KK(n;H)|$
of dimension $n-\frac{1}{2}(\# H)$.
\end{theorem}

\figer{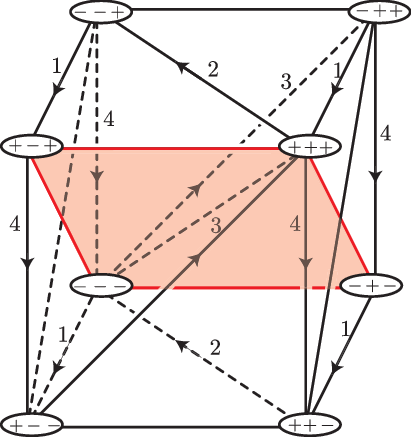}{0.5}{fig:kcpx2}{The action of $\hk'$ on $\KK(4;\emptyset)=\KK(4)$
induced by the exceptional strong inversion $h'_{q/p}$.
The shaded plane represents the fixed point set of $\hk'$.
The automorphism $\hk'$ of $\KK(4;H)$
for any non-empty symmetric subset $H$ of $\{1,2,3,4\}$
is the identity.}

\begin{proof}
(1) For each element $\rho\in\mathcal{O}(T)$,
let $-\rho$ be the element of $\mathcal{O}(T)$
obtained from $\rho$
by reversing the orientation of every edge of $T$.
Then we can observe that both of the
the strong inversions $h_{q/p}$ and $h_{q'/p}$
map the Seifert surface $F(\rho)$ to $F(-\rho)$ for every $\rho\in\mathcal{O}(T)$.
This means that both strong inversions 
map $F(\ee)$ to $F(-\ee)$
for every $\ee\in \{-1,1\}^{2n-1}$.
Hence both of them induce the automorphism $\hk$ of $MS(K)\cong \KK(n;H)$.
Since $\hk$ is determined by the restriction to $|\KK(n;H)|$ of 
the linear-automorphism $\vec{x}\mapsto-\vec{x}$ of $\RR^{2n-1}/W$,
the fixed point set $\Fix(\hk)\subset \KK(n,H)$ consists only of the origin.
On the other hand, the Seifert surfaces $F_+=F(+1,\cdots,+1)$
and $F_-=F(-1,\cdots,-1)$ are disjoint, as observed in the explanation of 
Proposition \ref{prop:structure-MS-special}.
So, if the vertices $[F_+]$ and $[F_-]$ of $MS(K)$ are distinct,
then they span an edge $e$.
The center of $e$ corresponds to the origin, and it
is the unique fixed point of $\hk$.
If $[F_+]=[F_-]$, then the vertex  
corresponds to the origin of $\KK(n:H)$, and it is the unique fixed point of $\hk$.
By using Proposition \ref{prop:structure-MS} 
(cf. \cite[Section 5]{Sakuma1994}),
we can observe that  $[F_+]=[F_-]$
if and only if $|a_i|=1$ either for all odd $i$ or for all even $i$.
Hence we obtain the assertion (1).

(2) We can observe that the exceptional strong inversion $h'_{q/p}$
maps the Seifert surface 
$F(\epsilon_1,\epsilon_2,\cdots,\epsilon_{2n-1})$
to $F(\epsilon_{2n-1},\cdots,\epsilon_2,\epsilon_1)$ up to isotopy
for any
$(\epsilon_1,\epsilon_2,\cdots,\epsilon_{2n-1})$
$\in$
$\{-1,1\}^{2n-1}$.
(To be precise, though the strong equivalence classes of exceptional strong inversions
are not unique as noted in Remark \ref{rem:strong-equivalence}(2),
the above holds for any exceptional strong inversion.) 
This implies that it induces the automorphism $\hk'$ of $MS(K)\cong \KK(n;H)$.
Thus $\Fix(\hk')$ is equal to the ball properly embedded in $|\KK(n;H)|\subset \RR^n/W$
obtained as the intersection of $|\KK(n;H)|$ with the
fixed point set of the linear automorphism of $\RR^{2n-1}/W$ induced by
the linear automorphism $\tilde{\hk}'$ of $\RR^{2n-1}$.
By the definition of $\tilde{\hk}'$,
the fixed point set (or the eigen space with eigen value $1$) of $\tilde{\hk}'$
is the $n$-dimensional space with basis
$\{\be_i+\be_{2n-i}\}_{1\le i\le n}$.
We can observe the set $\{\vv_i-\vv_{2n-i}\}_{1\le i\le n}$ is also a basis of
$\Fix(\tilde{\hk}')$.
On the other hand, since $H$ is symmetric,
there is a subset $J\subset \{1,2,\cdots, n\}$ of cardinality $d=\frac{1}{2}(\# H)$,
such that
$W=\kernel(\pi)$ has a basis 
$\{\vv_{i}, \vv_{2n+1-i}\}_{i\in J}$.
Hence, the subspace 
$W+\Fix(\tilde{\hk'})$ has basis 
$\{\vv_{i}, \vv_{2n+1-i}\}_{i\in J}\cup \{\vv_i-\vv_{2n-i}\}_{i\in J^c}$.
Thus $\dim(W+\Fix(\tilde{\hk'}))=2d+(n-d)=n+d$.
Since the fixed point set of the linear automorphism of $\RR^{2n-1}/W$
is equal to the subspace $(W+\Fix(\tilde{\hk'}))/W$, 
its dimension is equal to $(n+d)-2d=n-d$.
This completes the proof of (2).
\end{proof}

\section{Appendix: Table of enhanced equivariant genera of $2$-bridge knots up to 10 crossings}
By using Theorem \ref{thm:main},
we give a table of equivariant genera for $(K, h, \delta)$ with $K$
$2$-bridge knots up to $10$ crossings.
Let $K$ be the 2-bridge knot
$K(q/p)=K[2a_1,2b_1, 2a_2, 2b_2, \cdots, 2a_n,2n_n]$.
The data of the rationals of $2$-bridge knots 
are imported from the KnotInfo [KI], where
we have replaced $q$ by $p-q$ if $q$ is odd.
The 3rd column of Table shows the negative continued fractions, and
$\bar{n}$ means $-n$.
In the last column, we list the genus and equivariant genera:

{\tiny
$
\begin{array}{cccccccc}
&
\{&
g(K),&
g(K, h_{q/p},\tau_{q/p}),&
g(K, h_{q/p},\tau_{q/p}^c),&
g(K, h_{q'/p},\tau_{q'/p}),&
g(K, h_{q'/p},\tau_{q'/p}^c)&
\} \\
=&
\{&
n,&
n+\#\{i \ \vert\  |a_i|>1\},&
n+\sum_{i}^n( |b_i| -1),&
n+\#\{i \ \vert\  |b_i|>1\},&
n+\sum_{i}^n (|a_i| -1)&
\}\\
=&
\{&
n,&
n+  \#\{i \ \vert\  |a_i|>1\},&
2n - \#\{i \ \vert\  |b_i|=1\},&
n+   \#\{i \ \vert\  |b_i|>1\},&
2n - \#\{i \ \vert\  |a_i|=1\}&
\}
\end{array}
$
}

When $q=p-1$, $K$ is a torus knot and
has two equivariant genera, which both coincide with $g(K)$.
This case is marked as $\{g(K), {\rm \lq\lq torus"}\}$.
If non-tours $K$ is fibered or has the symmetry of 
$q^2\equiv 1 \pmod p$, equivariant genera coinciding with 
$g(K)$ are abbreviated,
 hence we mark the genus $g(K)$ and
attributes \lq\lq fib\rq\rq\  and/or \lq\lq sym\rq\rq\  accordingly.

\medskip

{\tiny
$
\begin{array}{|c|c|c|c|}
\hline
K&q/p&cont.frac&genera\\ \hline
3_1 & \text{2/3} & \{2,2\} & \{1,\text{torus}\} \\ 
 4_1 & \text{2/5} & \left\{2,\bar{2}\right\} & \{1,\text{fib}\} \\
 5_1 & \text{4/5} & \{2,2,2,2\} & \{2,\text{torus}\} \\
 5_2 & \text{4/7} & \{2,4\} & \{1,1,2,2,1\} \\
 6_1 & \text{2/9} & \left\{4,\bar{2}\right\} & \{1,2,1,1,2\} \\
 6_2 & \text{4/11} & \left\{2,\bar{2},\bar{2},\bar{2}\right\} & \{2,\text{fib}\} \\
 6_3 & \text{8/13} & \left\{2,2,\bar{2},\bar{2}\right\} & \{2,\text{fib}\} \\
 7_1 & \text{6/7} & \{2,2,2,2,2,2\} & \{3,\text{torus}\} \\
 7_2 & \text{6/11} & \{2,6\} & \{1,1,2,3,1\} \\
 7_3 & \text{4/13} & \{4,2,2,2\} & \{2,3,2,2,3\} \\
 7_4 & \text{4/15} & \{4,4\} & \{1,\text{sym},2,2\} \\
 7_5 & \text{10/17} & \{2,4,2,2\} & \{2,2,3,3,2\} \\
 7_6 & \text{12/19} & \left\{2,2,\bar{2},2\right\} & \{2,\text{fib}\} \\
 7_7 & \text{8/21} & \left\{2,\bar{2},\bar{2},2\right\} & \{2,\text{fib},\text{sym}\} \\
 8_1 & \text{2/13} & \left\{6,\bar{2}\right\} & \{1,3,1,1,2\} \\
 8_2 & \text{6/17} & \left\{2,\bar{2},\bar{2},\bar{2},\bar{2},\bar{2}\right\} & \{3,\text{fib}\} \\
 8_3 & \text{4/17} & \left\{4,\bar{4}\right\} & \{1,2,2,2,2\} \\
 8_4 & \text{14/19} & \left\{2,2,2,\bar{4}\right\} & \{2,2,3,3,2\} \\
 8_6 & \text{10/23} & \left\{2,\bar{4},\bar{2},\bar{2}\right\} & \{2,2,3,3,2\} \\
 8_7 & \text{14/23} & \left\{2,2,\bar{2},\bar{2},\bar{2},\bar{2}\right\} & \{3,\text{fib}\} \\
 8_8 & \text{16/25} & \left\{2,2,\bar{4},\bar{2}\right\} & \{2,3,2,2,3\} \\
 8_9 & \text{18/25} & \left\{2,2,2,\bar{2},\bar{2},\bar{2}\right\} & \{3,\text{fib}\} \\
 8_{11} & \text{10/27} & \left\{2,\bar{2},\bar{2},\bar{4}\right\} & \{2,2,3,3,2\} \\
 8_{12} & \text{12/29} & \left\{2,\bar{2},2,\bar{2}\right\} & \{2,\text{fib}\} \\
 8_{13} & \text{18/29} & \left\{2,2,\bar{2},\bar{4}\right\} & \{2,2,3,3,2\} \\
 8_{14} & \text{12/31} & \left\{2,\bar{2},\bar{4},\bar{2}\right\} & \{2,3,2,2,3\} \\
 9_1 & \text{8/9} & \{2,2,2,2,2,2,2,2\} & \{4,\text{torus}\} \\
 9_2 & \text{8/15} & \{2,8\} & \{1,1,2,4,1\} \\
 9_3 & \text{6/19} & \{4,2,2,2,2,2\} & \{3,4,3,3,4\} \\
 9_4 & \text{16/21} & \{2,2,2,6\} & \{2,2,3,4,2\} \\
 9_5 & \text{6/23} & \{4,6\} & \{1,2,2,3,2\} \\
 9_6 & \text{22/27} & \{2,2,2,2,4,2\} & \{3,4,3,3,4\} \\
 9_7 & \text{16/29} & \{2,6,2,2\} & \{2,2,3,4,2\} \\
 9_8 & \text{20/31} & \left\{2,2,\bar{4},2\right\} & \{2,3,2,2,3\} \\
 9_9 & \text{22/31} & \{2,2,4,2,2,2\} & \{3,4,3,3,4\} \\
 9_{10} & \text{10/33} & \{4,2,2,4\} & \{2,\text{sym},3,3\} \\
 9_{11} & \text{14/33} & \left\{2,\bar{2},2,2,2,2\right\} & \{3,\text{fib}\} \\
 9_{12} & \text{22/35} & \left\{2,2,\bar{2},4\right\} & \{2,2,3,3,2\} \\
 9_{13} & \text{10/37} & \{4,4,2,2\} & \{2,3,3,3,3\} \\
 9_{14} & \text{14/37} & \left\{2,\bar{2},\bar{2},4\right\} & \{2,2,3,3,2\} \\
 9_{15} & \text{16/39} & \left\{2,\bar{2},4,2\right\} & \{2,3,2,2,3\} \\
 9_{17} & \text{14/39} & \left\{2,\bar{2},\bar{2},\bar{2},\bar{2},2\right\} & \{3,\text{fib},\text{sym}\} \\
 9_{18} & \text{24/41} & \{2,4,2,4\} & \{2,2,4,4,2\} \\
 9_{19} & \text{16/41} & \left\{2,\bar{2},\bar{4},2\right\} & \{2,3,2,2,3\} \\
 9_{20} & \text{26/41} & \left\{2,2,\bar{2},2,2,2\right\} & \{3,\text{fib}\} \\
 9_{21} & \text{18/43} & \left\{2,\bar{2},2,4\right\} & \{2,2,3,3,2\} \\
 9_{23} & \text{26/45} & \{2,4,4,2\} & \{2,\text{sym},3,3\} \\
 \hline
\end{array}$
$
\begin{array}{|c|c|c|c|}
\hline
 9_{26} & \text{18/47} & \left\{2,\bar{2},\bar{2},2,2,2\right\} & \{3,\text{fib}\} \\
 9_{27} & \text{30/49} & \left\{2,2,\bar{2},\bar{2},\bar{2},2\right\} & \{3,\text{fib}\} \\
 9_{31} & \text{34/55} & \left\{2,2,\bar{2},\bar{2},2,2\right\} & \{3,\text{fib},\text{sym}\} \\
 10_1 & \text{2/17} & \left\{8,\bar{2}\right\} & \{1,4,1,1,2\} \\
 10_2 & \text{8/23} & \left\{2,\bar{2},\bar{2},\bar{2},\bar{2},\bar{2},\bar{2},\bar{2}\right\} & \{4,\text{fib}\} \\
 10_3 & \text{6/25} & \left\{4,\bar{6}\right\} & \{1,2,2,3,2\} \\
 10_4 & \text{20/27} & \left\{2,2,2,\bar{6}\right\} & \{2,2,3,4,2\} \\
 10_5 & \text{20/33} & \left\{2,2,\bar{2},\bar{2},\bar{2},\bar{2},\bar{2},\bar{2}\right\} & \{4,\text{fib}\} \\
 10_6 & \text{16/37} & \left\{2,\bar{4},\bar{2},\bar{2},\bar{2},\bar{2}\right\} & \{3,3,4,4,3\} \\
 10_7 & \text{16/43} & \left\{2,\bar{2},\bar{2},\bar{6}\right\} & \{2,2,3,4,2\} \\
 10_8 & \text{6/29} & \left\{4,\bar{2},\bar{2},\bar{2},\bar{2},\bar{2}\right\} & \{3,4,3,3,4\} \\
 10_9 & \text{28/39} & \left\{2,2,2,\bar{2},\bar{2},\bar{2},\bar{2},\bar{2}\right\} & \{4,\text{fib}\} \\
 10_{10} & \text{28/45} & \left\{2,2,\bar{2},\bar{6}\right\} & \{2,2,3,4,2\} \\
 10_{11} & \text{30/43} & \left\{2,2,4,\bar{4}\right\} & \{2,3,3,3,3\} \\
 10_{12} & \text{30/47} & \left\{2,2,\bar{4},\bar{2},\bar{2},\bar{2}\right\} & \{3,4,3,3,4\} \\
 10_{13} & \text{22/53} & \left\{2,\bar{2},2,\bar{4}\right\} & \{2,2,3,3,2\} \\
 10_{14} & \text{22/57} & \left\{2,\bar{2},\bar{4},\bar{2},\bar{2},\bar{2}\right\} & \{3,4,3,3,4\} \\
 10_{15} & \text{24/43} & \left\{2,4,\bar{2},\bar{2},\bar{2},\bar{2}\right\} & \{3,3,4,4,3\} \\
 10_{16} & \text{14/47} & \left\{4,2,2,\bar{4}\right\} & \{2,3,3,3,3\} \\
 10_{17} & \text{32/41} & \left\{2,2,2,2,\bar{2},\bar{2},\bar{2},\bar{2}\right\} & \{4,\text{fib}\} \\
 10_{18} & \text{32/55} & \left\{2,4,2,\bar{4}\right\} & \{2,2,4,4,2\} \\
 10_{19} & \text{14/51} & \left\{4,2,\bar{2},\bar{2},\bar{2},\bar{2}\right\} & \{3,4,3,3,4\} \\
 10_{20} & \text{16/35} & \left\{2,\bar{6},\bar{2},\bar{2}\right\} & \{2,2,3,4,2\} \\
 10_{21} & \text{16/45} & \left\{2,\bar{2},\bar{2},\bar{2},\bar{2},\bar{4}\right\} & \{3,3,4,4,3\} \\
 10_{22} & \text{36/49} & \left\{2,2,2,\bar{4},\bar{2},\bar{2}\right\} & \{3,3,4,4,3\} \\
 10_{23} & \text{36/59} & \left\{2,2,\bar{2},\bar{2},\bar{2},\bar{4}\right\} & \{3,3,4,4,3\} \\
 10_{24} & \text{24/55} & \left\{2,\bar{4},\bar{2},\bar{4}\right\} & \{2,2,4,4,2\} \\
 10_{25} & \text{24/65} & \left\{2,\bar{2},\bar{2},\bar{4},\bar{2},\bar{2}\right\} & \{3,3,4,4,3\} \\
 10_{26} & \text{44/61} & \left\{2,2,2,\bar{2},\bar{2},\bar{4}\right\} & \{3,3,4,4,3\} \\
 10_{27} & \text{44/71} & \left\{2,2,\bar{2},\bar{4},\bar{2},\bar{2}\right\} & \{3,3,4,4,3\} \\
 10_{28} & \text{34/53} & \left\{2,2,\bar{4},\bar{4}\right\} & \{2,3,3,3,3\} \\
 10_{29} & \text{26/63} & \left\{2,\bar{2},2,\bar{2},\bar{2},\bar{2}\right\} & \{3,\text{fib}\} \\
 10_{30} & \text{26/67} & \left\{2,\bar{2},\bar{4},\bar{4}\right\} & \{2,3,3,3,3\} \\
 10_{31} & \text{32/57} & \left\{2,4,\bar{2},\bar{4}\right\} & \{2,2,4,4,2\} \\
 10_{32} & \text{40/69} & \left\{2,4,2,\bar{2},\bar{2},\bar{2}\right\} & \{3,3,4,4,3\} \\
 10_{33} & \text{18/65} & \left\{4,2,\bar{2},\bar{4}\right\} & \{2,3,3,3,3\} \\
 10_{34} & \text{24/37} & \left\{2,2,\bar{6},\bar{2}\right\} & \{2,4,2,2,3\} \\
 10_{35} & \text{20/49} & \left\{2,\bar{2},4,\bar{2}\right\} & \{2,3,2,2,3\} \\
 10_{36} & \text{20/51} & \left\{2,\bar{2},\bar{6},\bar{2}\right\} & \{2,4,2,2,3\} \\
 10_{37} & \text{30/53} & \left\{2,4,\bar{4},\bar{2}\right\} & \{2,3,3,3,3\} \\
 10_{38} & \text{34/59} & \left\{2,4,4,\bar{2}\right\} & \{2,3,3,3,3\} \\
 10_{39} & \text{22/61} & \left\{2,\bar{2},\bar{2},\bar{2},\bar{4},\bar{2}\right\} & \{3,4,3,3,4\} \\
 10_{40} & \text{46/75} & \left\{2,2,\bar{2},\bar{2},\bar{4},\bar{2}\right\} & \{3,4,3,3,4\} \\
 10_{41} & \text{26/71} & \left\{2,\bar{2},\bar{2},\bar{2},2,\bar{2}\right\} & \{3,\text{fib}\} \\
 10_{42} & \text{50/81} & \left\{2,2,\bar{2},\bar{2},2,\bar{2}\right\} & \{3,\text{fib}\} \\
 10_{43} & \text{46/73} & \left\{2,2,\bar{2},2,\bar{2},\bar{2}\right\} & \{3,\text{fib}\} \\
 10_{44} & \text{30/79} & \left\{2,\bar{2},\bar{2},2,\bar{2},\bar{2}\right\} & \{3,\text{fib}\} \\
 10_{45} & \text{34/89} & \left\{2,\bar{2},\bar{2},2,2,\bar{2}\right\} & \{3,\text{fib}\} \\
 \hline
\end{array}
$}

\end{document}